\documentclass[reqno]{article}
\usepackage{amsmath,amssymb,cite}
\usepackage{amsthm}
\usepackage{graphics,epsfig,color}
\usepackage[mathscr]{euscript}
\usepackage{mathtools}
\usepackage{bbm}

\usepackage{comment}
\usepackage{enumerate}

\usepackage{graphicx}
\usepackage{graphics}
\usepackage{caption}
\usepackage{subcaption}
\usepackage{hyperref}
\usepackage{tikz}
\usetikzlibrary{arrows}

\usepackage{xcolor}

\usepackage{bm}

\theoremstyle{plain}
\newtheorem{theorem}{Theorem}[section]
\newtheorem{proposition}[theorem]{Proposition}
\newtheorem{lemma}[theorem]{Lemma}

\theoremstyle{definition}

\theoremstyle{remark}

\numberwithin{equation}{section}
\numberwithin{theorem}{section}

\newcommand{\de}{\textrm{d}}
\newcommand{\dee}{\mathrm{d}}

\newcommand{\ZZe}{\mathscr{Z}^\epsilon}
\newcommand{\Ze}{Z^\epsilon}
\newcommand{\Zpe}{Z^{'\epsilon}}
\newcommand{\YYe}{\mathscr{Y}^\epsilon}
\newcommand{\XXe}{\mathscr{X}^\epsilon}

\newcommand{\rhoe}{\rho^\epsilon}
\newcommand{\rhoep}{\rho^{'\epsilon}}
\newcommand{\trhoe}{\tilde{\rho}^\epsilon}
\newcommand{\trhoep}{\tilde{\rho}^{'\epsilon}}

\newcommand{\trhoepn}{\tilde{\rho}^{'\epsilon_n}}
\newcommand{\thetae}{\theta^\epsilon}
\newcommand{\tthetaep}{\tilde{\theta}^{'\epsilon}}

\newcommand{\tthetaepn}{\tilde{\theta}^{'\epsilon_n}}
\newcommand{\varrhoe}{\varrho^\epsilon}
\newcommand{\varthetae}{\vartheta^\epsilon}

\newcommand{\tZe}{\tilde{Z}^\epsilon}

\newcommand{\Ye}{Y^\epsilon}
\newcommand{\Ype}{Y^{'\epsilon}}

\newcommand{\tZed}{\tilde{Z}^\epsilon_{\Delta_\epsilon}}
\newcommand{\tZpe}{\tilde{Z}^{'\epsilon}}

\newcommand{\ue}{U^\epsilon}
\newcommand{\pe}{\Phi^\epsilon}
\newcommand{\ued}{U^\epsilon_{\Delta_\epsilon}}
\newcommand{\ped}{\Phi^\epsilon_{\Delta_\epsilon}}

\newcommand{\w}{w}
\newcommand{\y}{y}
\newcommand{\z}{z}
\newcommand{\Id}{\mathrm{Id}}
\newcommand{\nod}{\mathfrak{n}_{\Delta_\epsilon}}

\newcommand{\etq}{\epsilon^{-\frac{3}{4}}}
\newcommand{\etqp}{\epsilon^{\frac{3}{4}}}
\newcommand{\euq}{\epsilon^{\frac{1}{4}} }
\newcommand{\eum}{\epsilon^{\frac{1}{2}}}
\newcommand{\emq}{\epsilon^{-\frac{1}{4}}}
\newcommand{\emm}{\epsilon^{-\frac{1}{2}}}

\usepackage{accents}
\newcommand{\dbtilde}[1]{\accentset{\approx}{#1}}

\renewcommand{\epsilon}{\varepsilon}
\renewcommand{\tilde}{\widetilde}

\DeclareMathOperator{\Expect}{\mathsf{E}}

\begin{document}



\title{Long time fluctuations at critical parameter of Hopf's bifurcation \footnote{Declaration of interests: none. This research did not receive any specific grant from funding agencies in the public, commercial, or not-for-profit sectors. Michele Aleandri and Paolo Dai Pra are members of the Gruppo Nazionale per l’Analisi Matematica, la Probabilità e le loro Applicazioni (GNAMPA) of the Istituto Nazionale di Alta Matematica (INdAM).\\ Declaration of generative AI and AI-assisted technologies in the writing process: none.}} 


\author{M. Aleandri \footnote{LUISS University, Viale Romania 32, 00197 Rome, Italy. \textit{maleadri@luiss.it}}\, and  P. Dai Pra\footnote{Università di Verona,  Strada le Grazie, 15, 37134 Verona, Italy. \textit{paolo.daipra@univr.it} } }   
\date{}
\maketitle
\begin{abstract}
A dynamical system that undergoes a supercritical Hopf's bifurcation is perturbed by a multiplicative Brownian motion that scales with a small parameter $\epsilon$. The random fluctuations of the system at the critical point are studied when the dynamics starts near equilibrium, in the limit as $\epsilon$ goes to zero.  Under a space-time scaling the system can be approximated by a 2-dimensional process lying on the centre manifold  of the Hopf's bifurcation and a slow radial component together with a fast angular component are identified. Then the critical fluctuations are described by a “universal” stochastic differential equation whose coefficients are obtained taking the average with respect to the fast variable.
\end{abstract}

{\bf Keywords} Hopf's bifurcation, small noise perturbation, averaging principle, critical fluctuations.\\
{\bf Subjclass [2010]}
{	60H10, 
	34F05, 
	37G05, 
	60F99.
	



\section{Introduction}

The emergence of self-sustained periodic behavior is a widespread phenomenon in nature: this has stimulated the mathematical modeling that could account for this behavior. For deterministic dynamical systems, in particular for those described by finite dimensional ordinary differential equations, the presence of a {\em Hopf bifurcation} is a feature producing oscillations in a ``universal'' way. 

Dynamical systems of this sort 
find application in various areas such as the investigation of aircraft panel flutter in high supersonic conditions \cite{dowell1970panel}, the dynamics of individual neurons \cite{izhikevich2007dynamical} and the physiological reactions of cells to external stimuli \cite{borisuk1998bifurcation}.\\
Naturally, a question arises in these classical models: What happens if a random perturbation is added to the dynamics? A theory on bifurcations in random dynamical systems was developed by Ludwig Arnold \cite{arnold1996toward} and Peter Baxendal \cite{baxendale1994stochastic}. More recently in \cite{doan2018hopf} the dynamics of a two-dimensional ordinary differential equation in the normal form of Hopf's bifurcation perturbed by additive white noise is studied stating that, for a generic deterministic Hopf's bifurcation similar local dynamics behaviour occurs. In the small noise regime the  theory due to Freidlin and Wentzell, \cite{freidlin1998random}, shows that the the stochastic process behaves as the deterministic counterpart, but close to critical events such as bifurcations, the system can show great sensitivity to small scale random perturbations, see \cite{juel1997effect} for an experimental and numerical study. 

We also remark that the small noise asymptotics are close in spirit to large scale asymptotics for mean field interacting stochastic systems (see e.g. \cite{collet2015collective}). In that context a deterministic dynamical system emerges in the infinite particle limit, and finite size effects correspond to the small noise regime. \\
In this paper  we consider a $n$-dimensional dynamical system that exhibits a supercritical Hopf's bifurcation, and is then perturbed by a multiplicative noise that scales with a small parameter $\epsilon$. We set the system at the critical point, i.e. at the Hopf bifurcation, and study the random fluctuations of the dynamics starting near equilibrium, in the limit as $\epsilon \rightarrow 0$. A space-time scaling allows to approximate the system by a $2$-dimensional process lying on the centre manifold of the Hopf's bifurcation.

In this regime two scales are identified: a slow radial component and a fast angular component. In the limit, as $\epsilon$ goes to zero, an averaging effect is observed:  the slow variable converges to the solution of a stochastic differential equation  whose coefficients are obtained by averaging with respect to the fast variable. The study of averaging principle in various dynamical systems dates back to the works of Khasminskii and others, summarized in, e.g., Freidlin and Wentzell \cite{freidlin1998random}, Kabanov and Pergamenshchikov \cite{kabanov2003two}. In this work the average process is obtained following the Strook-Varadhan approach to martingale problem, also developed in detailed in \cite{dai2019dynamics} where an averaging principle is proved.\\

The paper is structured as follows. In section \ref{sec:model} the model is presented and heuristic arguments about the limit process are presented. Section \ref{sec:normal form} is devoted to obtain the normal form of the rescaled stochastic process and in section \ref{sec:reduction} the reduction principle is proved. The $2$-dimensional process and its limit is studied in section \ref{sec:2dimensional} and section \ref{sec:proofMain} concludes the paper with the proof of the main result. Appendix \ref{appendix} contains the explicit calculations to define the transformation that allows to obtain the normal form of the deterministic process.

\subsection{Notation}
Small letters are used for the deterministic processes, i.e. $x(t)\in\mathbb{R}^n$, $z(t)\in\mathbb{R}^2$, and $y(t)\in\mathbb{R}^{n-2}$. Capital letters and Greek letter are used for the stochastic processes, i.e. $\mathscr{X}^\epsilon(t),X^\epsilon(t)\in\mathbb{R}^n$, $\mathscr{Z}^\epsilon(t),Z^\epsilon(t)\in\mathbb{R}^2$,  $\mathscr{Y}^\epsilon(t),Y^\epsilon(t)\in\mathbb{R}^{n-2}$ and $(\varrhoe(t),\varthetae(t)),(\rho(t),\theta(t)) \in(0,\infty)\times\mathbb{R}$, $t\geq0$. Given a vector $x$,  set $\|x\|:=(\sum_{i}x_i^2)^{\frac{1}{2}}$, and for a matrix $A$, set $\|A\|:=\big(\sum_i\sum_jA_{ij}^2\big)^{\frac{1}{2}}$. Given a matrix $A$, then $A^{\mathsf{T}}$ is the transposed matrix of $A$. Given a function $f:\mathbb{R}^n\to\mathbb{R}^m$, call $Df=\big(Df_i\big)_{i=1}^m$ the vector of the $n$-dimensional gradients of the components of $f$, $D^2f=\big(D^2f_i\big)_{i=1}^m$ the vector of the $n\times n$-hessian matrices of the components of $f$ and $\mathrm{Tr} D^2f=\big(\mathrm{Tr}D^2f_i\big)_{i=1}^m$ the vector of the traces.


\section{The model and main result}\label{sec:model}

Consider a family of dynamical systems, indexed by a real parameter $\mu$, given by the ordinary differential equation in $\mathbb{R}^n$, $n\in\mathbb{N}$,
\begin{equation}\label{eq:ODE}
	\de x(t)= b(x(t),\mu) \de t,
\end{equation}

with initial condition $x(0)=x_0$. Suppose that

\begin{itemize}
	\item[\textbf{H1.1}] $\forall\mu \in \mathbf{R}$, $b(\cdot,\mu):\mathbb{R}^n\to\mathbb{R}^n$ is a $C^4$ function;
	\item[\textbf{H1.2}] the system \eqref{eq:ODE} has a critical point $x_{\mu_c}$  for $\mu=\mu_c$, i.e. $b(x_{\mu_c},\mu_c)=0$;
	\item[\textbf{H1.3}] the Jacobian matrix $\textrm{D} b(x_{\mu_c},\mu_c)$ has 2 simple pure imaginary eigenvalues and $n-2$ eigenvalues with negative real part;
	\item[\textbf{H1.4}]	 $\frac{\de}{\de \mu}[\operatorname{Re} \lambda(\mu)]_{\mu=\mu_c} \neq 0 $, where  $\lambda(\mu)$, $\bar{\lambda}(\mu)$ are the eigenvalues of the Jacobian matrix $\textrm{D} b(x_{\mu},\mu)$\footnote{The existence of the critical point $x_\mu$ is assured by \textbf{H1.1}, \textbf{H1.2} and \textbf{H1.3}, see \cite[Theorem~2]{perko2013differential}.}.
\end{itemize}
These are the standard hypothesis to observe a Hopf's bifurcation as stated in  \cite[Theorem~2]{perko2013differential}. 
The eigenvalues of $Db(x_{\mu},\mu)$ intersect the imaginary axis at $\mu=\mu_c$, then the dimensions of the stable and unstable manifolds of $\mathbf{x}_\mu$ will undergo alteration, leading to a modification in the local phase portrait of \eqref{eq:ODE} as $\mu$ traverses the bifurcation point $\mu_c$. Generally speaking, a Hopf's bifurcation occurs when a periodic orbit emerges due to the changing stability of the equilibrium point $x_{\mu_c}$. Moreover it is called \emph{supercritical} if for $\mu\leq 0$ the origin is an asymptotically stable equilibrium and for $\mu>0$ an attractive limit cycle appears and the origin becomes unstable. Figure \ref{fig:Hopfd3} shows an example of supercritical bifurcation in dimension 3. 
\begin{figure}
	\centering
	\begin{subfigure}{0.3\textwidth}
		\centering
		\includegraphics[width=\textwidth]{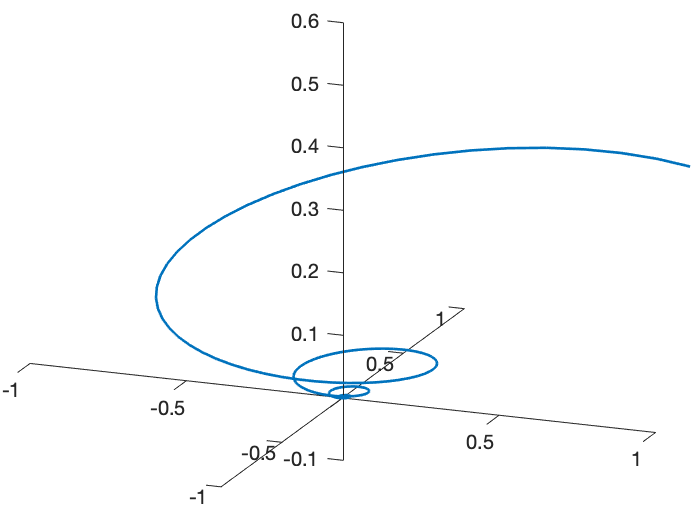}
		\caption{$\mu<0$}
	\end{subfigure}
	\begin{subfigure}{0.3\textwidth}
		\centering
		\includegraphics[width=\textwidth]{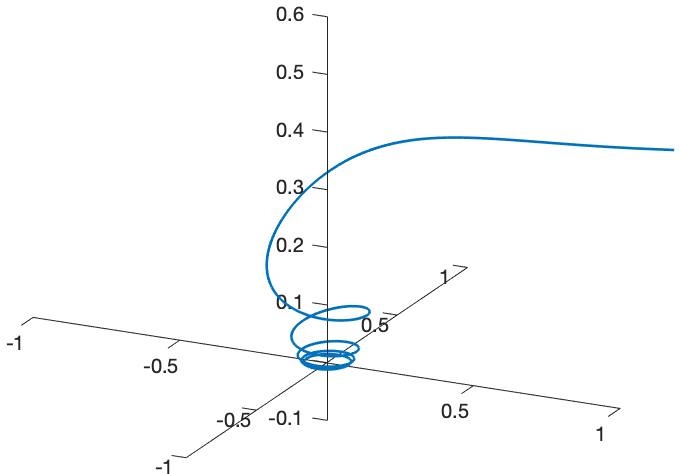}
		\caption{$\mu=0$}
	\end{subfigure}
	\begin{subfigure}{0.3\textwidth}
		\centering
		\includegraphics[width=\textwidth]{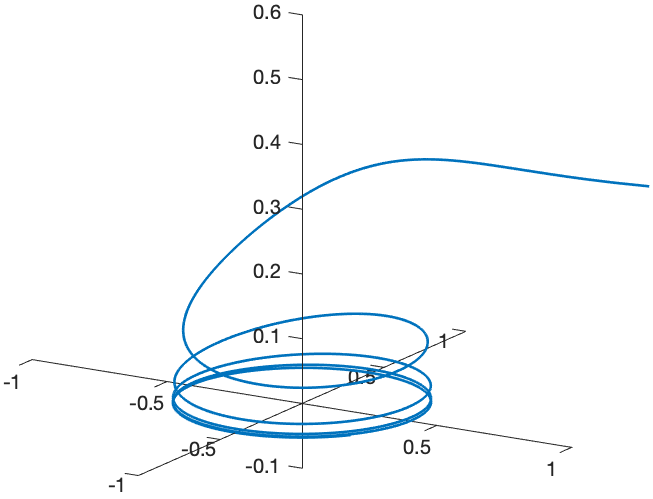}
		\caption{$\mu>0$}
	\end{subfigure}
	\caption{A supercritical Hopf's bifurcation in dimension 3.}
	\label{fig:Hopfd3}
\end{figure}
\noindent
Then make the following assumption 
\begin{itemize}
	\item[(\textbf{H1}):] Hypothesis \textbf{H1.1}, \textbf{H1.2}, \textbf{H1.3} and \textbf{H1.4}  hold and the Hopf's bifurcation is supercritical.
\end{itemize}
For the rest of the paper, without loss of generality, take $x_{\mu_c}=0$, $\mu_c=0$. \\

In this paper we apply a random perturbations of the system above.

Let be $\sigma:\mathbb{R}^n\to\mathbb{R}^n\times\mathbb{R}^m$ a positive smooth function and $\epsilon>0$, then consider the following SDE:

\begin{equation}\label{eq:SDE}
	\de \mathscr{X}^\epsilon(t)= b(\mathscr{X}^\epsilon(t),\mu) \de t + \sqrt{\epsilon} \sigma(\mathscr{X}^\epsilon(t))\de B(t),
\end{equation}
with initial condition $\mathscr{X}^\epsilon(0)=X_0^\epsilon$ and $B$ is an $m$-dimensional Brownian motion.  Under suitable hypothesis of regularity of the function $\sigma$, that are postponed to section \ref{subsec:main result}, the solution $\mathscr{X}^\epsilon$ of \eqref{eq:SDE} converges in distribution, on any time interval $[0,T]$ with $T>0$ and as $\epsilon$ goes to 0, to the solution $x$ of \eqref{eq:ODE} provided that $X_0^\epsilon$ converges to $x_0$, see \cite{freidlin1998random}. \\

In particular, if $X_0^\epsilon$ converges to $0$, then the whole process $\mathscr{X}^\epsilon$ converges to the constant $0$. The aim of this paper is to investigate the fluctuations around this limiting behavior, at the critical point $\mu = \mu_c=0$. This will require a time scaling, which captures the long time behavior of the original process, as well as a space scaling which ''amplifies'' the process around zero. This procedure will be first illustrated in next subsection, at a heuristic level, for a two dimensional system exhibiting a Hopf bifurcation in the so-called {\em normal form}.

\subsection{Heuristic for $2$-dimensional Hopf's bifurcation}\label{sec:heuristic}
A simple example of Hopf's bifurcation is given by the following $2$-dimensional system of equations
\begin{equation}\label{eq:normalform}
	\begin{aligned}
		\de {z}_1&=\Big(-z_2 -z_1(z_1^2+z_2^2-\mu) \Big)\de t,\\
		\de {z}_2&=\Big(z_1 -z_2(z_1^2+z_2^2-\mu) \Big)\de t.
	\end{aligned}
\end{equation}
As showed in \cite{perko2013differential}, the origin is a stable or an unstable focus of this nonlinear system if $\mu<0$ or if $\mu>0$ respectively. For $\mu=0$, $D f(0,0)$ has a pair of pure imaginary eigenvalues, and the origin is a stable focus. For $\mu>0$ a stable limit cycle $ \Gamma_\mu: \gamma_\mu(t)=\sqrt{\mu}(\cos t, \sin t)^T $ emerges. Suppose that the dimension of the perturbing Brownian motion is 
$m=2$ and that the diffusion coefficient in \eqref{eq:SDE} is constant, i.e. 
$$\sigma(x)=\begin{pmatrix} 	\sigma_{11} & \sigma_{12} \\ \sigma_{21} & \sigma_{22} \end{pmatrix}\in\mathbb{R}^{2\times2}.$$
Under these hypothesis, when $\mu=0$, equation \eqref{eq:SDE} can be written as:
\begin{equation}\label{SDE-2normal form}
	\begin{aligned}
		\de \ZZe_1&=\big(-\ZZe_2 -\ZZe_1\big((\ZZe_1)^2+(\ZZe_2)^2\big)\big)\de t +\sqrt{\epsilon}\big(\sigma_{11}\de B_1+\sigma_{12}\de B_2\big),\\ 
		\de \ZZe_2&=\big(\ZZe_1 -\ZZe_2\big((\ZZe_1)^2+(\ZZe_2)^2\big)\big)\de t+ \sqrt{\epsilon}\big(\sigma_{21}\de B_1+\sigma_{22}\de B_2\big). 
	\end{aligned}
\end{equation}

Note that the particular form of the drift allows the use of a ``compact containment'' argument, to show that \eqref{SDE-2normal form} has a unique global solution.
It is convenient to convert to polar coordinates performing the following change of coordinate:
\begin{equation*}
	\begin{aligned}
		\mathfrak{o}:&(0,\infty)\times\mathbb{R}\to \mathbb{R}^2\setminus\{(0,0)\}\\
		&\mathfrak{o}(\rho,\theta)=\begin{pmatrix*}\rho\cos\theta\\\rho\sin\theta\end{pmatrix*}=\begin{pmatrix*}z_1\\z_2\end{pmatrix*}
	\end{aligned}
\end{equation*}
Equation \eqref{eq:normalform} can be rewritten as $\begin{cases} \dot{\rho}=-\rho^3\\ \dot{\theta}=1 \end{cases}$, when $\mu=0$ and it is immediate to observe that the radial coordinate converges to the zero. Applying the same change of coordinates to the stochastic process $\mathfrak{o}(\varrhoe,\varthetae)=\left(\begin{smallmatrix}\varrhoe\cos\varthetae \\ \varrhoe\sin\varthetae \end{smallmatrix}\right) = \left(\begin{smallmatrix} \ZZe_1 \\ \ZZe_2 \end{smallmatrix}\right)$, then equation \eqref{SDE-2normal form} in the new variables reads:
\begin{align}\label{sde:polar2_1}
	&\begin{aligned}
		\de \varrhoe =&\frac{1}{\varrhoe}\Big[-(\varrhoe)^4+\frac{\epsilon}{2}(\sigma_{11}^2+\sigma_{12}^2) (\sin\varthetae)^2+\frac{\epsilon}{2}(\sigma_{21}^2+\sigma_{22}^2)(\cos\varthetae)^2 \\ 
		&\quad + \epsilon(\sigma_{11}\sigma_{21}+\sigma_{21}\sigma_{22})\sin\varthetae\cos\varthetae\Big]\de t \\
		&+\sqrt{\epsilon}\big(\sigma_{11}\cos\varthetae+\sigma_{21}\sin\varthetae\big)\de B_1 + \sqrt{\epsilon}\big(\sigma_{12}\cos\varthetae+\sigma_{22}\sin\varthetae\big)\de B_2;
	\end{aligned}\\\label{sde:polar2_2}
	&\begin{aligned}
		\de \varthetae =& \Big\{ 1-\epsilon\Big[\frac{1}{(\varrhoe)^2}\sin\varthetae\cos\varthetae\big(\sigma_{11}^2+\sigma_{12}^2-\sigma_{21}^2-\sigma_{22}^2\big)\\
		&\quad + \frac{1}{(\varrhoe)^2}\big((\cos\varthetae)^2-(\sin\varthetae)^2\big)(\sigma_{11}\sigma_{21}+\sigma_{12}\sigma_{22})\Big]\Big\}\de t\\ 
		&+\sqrt{\epsilon}\frac{1}{\varrhoe}\big(\sigma_{21}\cos\varthetae-\sigma_{11}\sin\varthetae\big)\de B_1 + \sqrt{\epsilon}\frac{1}{\varrhoe}\big(\sigma_{22}\cos\varthetae-\sigma_{12}\sin\varthetae\big)\de B_2.
	\end{aligned}
\end{align}
Note that in these equations we must deal with the singularity at $\rho=0$; this will be done rigorously by stopping the process as $\varrhoe$ enters a suitable neighborhood of zero. We now introduce the following space-time rescaling:

\begin{equation*}
	\rhoe(t):=\epsilon^{-1/4}\varrhoe(\epsilon^{-1/2}t),\qquad\thetae(t):=\varthetae(\epsilon^{-1/2}t).
\end{equation*}
Equations \eqref{sde:polar2_1} and \eqref{sde:polar2_2} become:
\begin{align}\label{sde:polar2_1e}
	&\begin{aligned}
		\de \rhoe =\frac{1}{\rhoe}&\Big[-(\rhoe)^4+\frac{1}{2}(\sigma_{11}^2+\sigma_{12}^2) (\sin\thetae)^2+\frac{1}{2}(\sigma_{21}^2+\sigma_{22}^2)(\cos\thetae)^2 \\ 
		&\quad + (\sigma_{11}\sigma_{21}+\sigma_{21}\sigma_{22})\sin\thetae\cos\thetae\Big]\de t + \de M_\rho^\epsilon;
	\end{aligned}\\ \label{sde:polar2_2e}
	&\begin{aligned}
		\de \thetae  = \Big\{\epsilon^{-1/2}&-\frac{1}{(\rhoe)^2}\Big[\sin\thetae\cos\thetae(\sigma_{11}^2-\sigma_{21}^2+\sigma_{12}^2-\sigma_{22}^2)\\ 
		&+\big((\cos\thetae)^2-(\sin\thetae)^2\big)(\sigma_{11}\sigma_{21}+\sigma_{12}\sigma_{22})\Big]\Big\}\de t+ \de M_\theta^\epsilon,
	\end{aligned}
\end{align}
where
\begin{align*}
	&M_\rho^\epsilon(t)  = \int_{0}^{t}\Big(\sigma_{11}\cos\thetae(r)+\sigma_{21}\sin\thetae(r)\Big)\de B_1(r)+\int_{0}^{t}\Big(\sigma_{12}\cos\thetae(r)+\sigma_{22}\sin\thetae(r)\Big)\de B_2(r);\\
	&\begin{aligned}
		M_\theta^\epsilon(t)  = \int_{0}^{t}\frac{1}{\rhoe(r)}\Big[\sigma_{21}\cos\thetae(r)- &\sigma_{11}\sin\thetae(r)\Big]\de B_1(r)\\ &+\int_{0}^{t}\frac{1}{\rhoe(r)}\Big[\sigma_{22}\cos\thetae(r)-\sigma_{12}\sin\thetae(r)\Big]\de B_2(r).
	\end{aligned}
\end{align*}

This scaling reveals the existence of two time scales, as the angular variable moves much faster than the radial one.
Letting $\epsilon$ going to zero the fast variable is expected to out; as the angular velocity is essentially constant, the averaging measure should be the uniform distribution on $[0,2\pi]$, so the coefficients in the dynamics of the radial variable must be averaged with respect the uniform measure, as we now illustrate.\\
The function $a:(0,\infty)\times\mathbb{R}\to \mathbb{R}$ given by
\begin{equation*}
	a(\eta,\phi)=\frac{1}{\eta}\Big[-\eta^4+ \frac{1}{2}(\sigma_{11}^2+\sigma_{12}^2) (\sin\phi)^2+\frac{1}{2}(\sigma_{21}^2+\sigma_{22}^2)(\cos\phi)^2 + (\sigma_{11}\sigma_{21}+\sigma_{21}\sigma_{22})\sin\phi\cos\phi\Big] ,
\end{equation*} 
which is the drift of the radial variable, should be replaced in the limit by
the function $\bar{a}:(0,\infty)\times\mathbb{R}\to \mathbb{R}$ given by
\begin{equation*}
	\bar{a}(\eta)=\frac{1}{2\pi}\int_{0	}^{2\pi}a(\eta,\phi)\de \phi= \frac{1}{\eta}\Big[-\eta^4+\frac{1}{4}\big(\sigma_{11}^2+\sigma_{12}^2+\sigma_{21}^2+\sigma_{22}^2\big)\Big].
\end{equation*}
Similarly, the martingale $M_\rho^\epsilon$, whose quadratic variation is given by
\begin{equation*}
	\langle M_\rho^\epsilon\rangle(t)= \int_{0}^{t}w\big(\thetae(r)\big)\de r,
\end{equation*}
with
$$w(\phi)=(\sigma_{11}\cos \phi+\sigma_{21}\sin \phi)^2+(\sigma_{21}\cos \phi+\sigma_{22}\sin \phi)^2,$$
is expected to converge to a martingale
$\bar{M}$, with quadratic variation:
\begin{equation*}
	\langle \bar{M}\rangle(t)=\int_{0}^{t}\bar{w}\de r = \int_{0}^{t}\frac{1}{2\pi} \int_0^{2\pi} w(\phi) \de \phi\de r  = \frac{1}{2}(\sigma_{11}^2+\sigma_{12}^2+\sigma_{21}^2+\sigma_{22}^2)t.
\end{equation*}

Summing up, we consider the following stochastic differential equation
\begin{equation}\label{eq:limit process2d}
	\de \bar{\rho}(t)=\bar{a}(\bar{\rho}(t))\de t + \de \bar{M}(t),
\end{equation}
with initial condition $\bar{\rho}(0)=\bar{\rho}_0$.  \\

Let $\rho^\epsilon,\bar{\rho}$ be the solutions to the equations \eqref{sde:polar2_1e}, \eqref{eq:limit process2d}, respectively, with initial conditions $\rho_0^\epsilon$, $\bar{\rho}_0$. The main result of this paper shows that, as $\epsilon\downarrow0$, if $\rho_0^\epsilon$ converges weakly to $\bar{\rho}_0$, then $\rho^\epsilon$ converges weakly to $\bar{\rho}$. \\

\subsection{Main result}\label{subsec:main result}

In this subsection the main result of this paper is stated after adding to (\textbf{H1}) a list of assumptions on the diffusion matrix $\sigma$.

Letting $A=\textrm{D} b(0,0)$, under hypothesis (\textbf{H1}) there exists an invertible matrix $C$ such that 
\begin{equation} \label{def:matrixC}
	B=C^{-1}A C = \begin{pmatrix}
		Q & 0 \\ 0 & P  
	\end{pmatrix}, 
\end{equation}
where $Q\in\mathbb{R}^{2\times2}$ has 2 pure imaginary eigenvalues , $P\in\mathbb{R}^{n-2\times n-2}$ has $n-2$ eigenvalues with negative real part.  

With no loss of generality we can assume $Q$ to be of the form
\[
Q = \left( \begin{array}{cc} 0 & - \lambda_0 \\ \lambda_0 & 0 \end{array} \right).
\]

Performing the change of coordinate 
\begin{equation}\label{function c}
	\begin{aligned}
		\mathfrak{c}:\ \mathbb{R}^n&\to\mathbb{R}^n\\
		x\, \, &\to C^{-1}x=(z,y)^{\mathsf{T}} 
	\end{aligned}
\end{equation}
system \eqref{eq:ODE} can be written as
\begin{align}\label{eq:Carr1}
	\de z (t) =& \Big[Qz(t)+f\big(z(t),y(t)\big)\Big]\de t, \\\label{eq:Carr2}
	\de y (t) =& \Big[Py(t)+g\big(z(t),y(t)\big)\Big]\de t,
\end{align}
where $f:\mathbb{R}^n\to\mathbb{R}^2$, $g:\mathbb{R}^n\to\mathbb{R}^{n-2}$ are two functions such that $f(0,0)=0$, $g(0,0)=0$, $\textrm{D}f(0,0)=\underbar{0}$ and $\textrm{D}g(0,0)=\underbar{0}$, where $\underbar{0}$ is the null matrix. \\

We make the following assumptions  on the diffusion coefficient $\sigma$:
\[
(\textbf{H2}):= \left\lbrace
\begin{tabular}{@{}p{10cm}p{10cm}} 
	\textbf{H2.1}\quad  the function $\sigma:\mathbb{R}^n\to\mathbb{R}^{n\times m}$ is a $C^4$ function; \\[3pt]
	\textbf{H2.2}\quad  $\sigma(0,0)\neq \underbar{0}$.
\end{tabular}\right.
\]

Performing the change of coordinate  $\mathfrak{c}\big(\XXe\big)=\big(\ZZe,\YYe\big)^{\mathsf{T}}$ the system \eqref{eq:SDE} can be written as
\begin{align}\label{eq:SDE2dZ}
	\de \ZZe (t) =& \Big[Q\ZZe(t)+f\big(\ZZe(t),\YYe(t)\big)\Big]\de t + \sqrt{\epsilon}\sigma_Q\big(\ZZe(t),\YYe(t)\big)\de B(t),\\\label{eq:SDEn-2dY}
	\de \YYe (t) =& \Big[P\YYe(t)+g\big(\ZZe(t),\YYe(t)\big)\Big]\de t + \sqrt{\epsilon}\sigma_P\big(\ZZe(t),\YYe(t)\big)\de B(t),
\end{align}
where  $\sigma_Q:\mathbb{R}^{2}\to\mathbb{R}^{m}$ and  $\sigma_P:\mathbb{R}^{n-2}\to\mathbb{R}^{m}$ are given by
\[
\left(\begin{array}{c} \sigma_Q(z,y) \\ \sigma_P(z,y) \end{array} \right) = C^{-1} \sigma(C(z,y)).
\]

We let also $\bar{\sigma}=\sigma_Q(0,0)$ and $\bar{\sigma}_{ij}$ its components. 

\begin{theorem}\label{theor:Main}
	Suppose that $b:\mathbb{R}^n\to\mathbb{R}^n$ satisfies assumptions (\textbf{H1}), and $\sigma:\mathbb{R}^n\to\mathbb{R}^{n\times n}$ satisfies assumption (\textbf{H2}). For any $\epsilon>0$, let $(\XXe(t))_{t\geq0}$ be the solution of the SDE \eqref{eq:SDE} with initial condition $x^\epsilon(0)$. Suppose that $\emq x^\epsilon(0)\to x(0)\neq 0$.
	For any $T>0$, define the process 
	
	\begin{equation} \label{Ze}
		\Big(\Ze(t),\Ye(t)\Big)_{t\in[0,T]}=\Big(\mathfrak{c}\big(\emq\XXe(\emm t)\big)\Big)_{t\in[0,T]},
	\end{equation}
	where the function $\mathfrak{c}$ is defined in \eqref{function c}. Then, as $\epsilon\downarrow0$,
	\begin{equation*}
		\Big(\rhoe(t)\Big)_{t\in[0,T]}=\Big(\sqrt{(\Ze_1(t))^2+(\Ze_2(t))^2}\Big)_{t\in[0,T]} \mbox{weakly converges to }  \Big(\bar{\rho}(t)\Big)_{t\in[0,T]},
	\end{equation*}
	where $\bar{\rho}$ is the solution to the SDE
	\begin{equation}\label{eq:limit process2ndimensional}
		\dee \bar{\rho}=\bar{b}(\bar{\rho})\dee t + \sqrt{\frac{1}{2}\big(\Sigma_{1}^2+\Sigma_{2}^2\big)}\dee \tilde{B},
	\end{equation}
	with $\bar{b}(\eta)= \frac{1}{\eta}[-\eta^4+\frac{1}{4}(\Sigma_{1}^2+\Sigma_{2}^2)]$, initial condition $\bar{\rho}_0 = \sqrt{z_1(0)^2+z_2(0)^2}$, where $\big(z(0),y(0)\big)=\mathfrak{c}\big(x(0)\big)$, and $\Sigma_i^2=\sum_{j=1}^n\bar{\sigma}_{ij}^2$, $i=1,2$.
\end{theorem}

The proof of Theorem \ref{theor:Main} is divided into several steps. In the first part a  “normal form" for the stochastic differential equation is written applying the local diffeomorphisms defined for the deterministic processes. The second step consists in the reduction of the $n$-dimensional problem to a $2$-dimensional problem. In the third part the global existence for the two dimensional stochastic process and its convergence, when $\epsilon$ goes to zero, are proved. Last step is to put the results together to conclude the proof.   \\


\section{Normal form of the rescaled process} \label{sec:normal form}
The theory of normal forms, see \cite{wiggins1996introduction}, when $\|(z,y)^\mathsf{T}\|<r$ for $r$ small enough allows to define a local diffeomorphism 
\begin{align}\label{diff:p}
	\mathfrak{p}:\ \mathbb{R}^n&\to\mathbb{R}^n\\
	\begin{pmatrix*}z' \\ y \end{pmatrix*}&\to \begin{pmatrix*} z' + p(z',y) \\ y \end{pmatrix*}=\begin{pmatrix*} z \\ y \end{pmatrix*} \nonumber
\end{align}
such that the system of equations \eqref{eq:Carr1},\eqref{eq:Carr2} can be written as
\begin{align}\label{eq:Carr1prime}
	\de z'(t) = & \Big[ Qz'(t)+f^{(3)}\big(z'(t),y(t)\big)+f^{(2)}\big(y(t)\big) + r_f\big(z'(t),y'(t)\big) \Big]\de t,\\ \label{eq:Carr2prime}
	\de y(t) = & \Big[Py(t)+g'\big(z'(t),y(t)\big)\Big]\de t, 
\end{align}
where 
\begin{enumerate}[i)]
	\item the components of the function $p(z',y)$ are polynomials of degree $2$ with terms of type $z_i'z_j'$, $z_i'y_k$, $i,j=1,2$ and $k=1,\ldots,n-2$; 
	\item the components of the function $f^{(3)}$ are polynomials of degree $3$;
	\item the components of the function $f^{(2)}$ are polynomials of degree $2$;
	\item the function $r_f$ contains the higher order terms, namely $D^\alpha r_f(0,0)=0$, for all multi-index $\alpha$ with $|\alpha|\leq 3$;
	\item the function $g'$ is the function $g$ evaluated in the new variable $z'$.
\end{enumerate}
The detailed computations to obtain equations \eqref{eq:Carr1prime} and \eqref{eq:Carr2prime} are postponed to Appendix \ref{appendix}. \\
The system of equations \eqref{eq:Carr1prime},\eqref{eq:Carr2prime} admits a centre manifold $\mathcal{W}_c$ described by a $C^2$ function $h:\mathbb{R}^2\to\mathbb{R}^{n-2}$  with $D^\alpha h(0)=0$, for  $|\alpha|\leq 1$,  see \cite{carr2012applications}. The dynamics on the centre manifold can be described locally around the origin by the equation

\begin{equation}
	\label{eq:Dcentre}
	\de \tilde{z} = \Big[ Q\tilde{z}(t)+f^{(3)}\Big(\tilde{z}(t),h\big(\tilde{z}(t)\big)\Big) \Big]\de t.
\end{equation}
Equation \eqref{eq:Dcentre} is two-dimensional then, applying directly the argument in \cite{wiggins1996introduction} regarding the normal form of the Hopf's bifurcation, there is a local diffeomorphism
\begin{align}\label{diff:q}
	\mathfrak{q}:\ \mathbb{R}^2&\to\mathbb{R}^2\\
	\dbtilde{z}\ & \to  \dbtilde{z} + q(\dbtilde{z}) = \tilde{z} \nonumber
\end{align}
such that the variable $\dbtilde{z}$ satisfies the equation
\begin{equation}\label{eq:z1z2normalform}
	\begin{aligned}
		\de \dbtilde{z}_1(t) &=  \Big[ -\lambda_0\dbtilde{z}_2(t)- \dbtilde{z}_1(t)\big(\dbtilde{z}_1(t)^2+\dbtilde{z}_2(t)^2\big)+f^{(4)}_1\big(\dbtilde{z}(t)\big) \Big]\de t,\\
		\de \dbtilde{z}_2(t) &=  \Big[ +\lambda_0\dbtilde{z}_1(t)- \dbtilde{z}_2(t)\big(\dbtilde{z}_1(t)^2+\dbtilde{z}_2(t)^2\big)+f^{(4)}_2\big(\dbtilde{z}(t)\big) \Big]\de t,
	\end{aligned}
\end{equation}
where the function $f^{(4)}=(f^{(4)}_1,f^{(4)}_2)^{\mathsf{T}}$ contains the higher order terms, namely $D^\alpha f^{(4)}(0)=0$, for all multi-index $\alpha$ with $|\alpha|\leq 3$.\\

Now, the local diffeomorphisms $\mathfrak{p}$ and $\mathfrak{q}$, defined in \eqref{diff:p} and \eqref{diff:q} respectively, are applied
to obtain a stochastic normal form of equation \eqref{eq:SDE}. 
\\

First, equations \eqref{eq:SDE2dZ} and \eqref{eq:SDEn-2dY} give: 
\begin{align}\label{eq:timescaling1}
	\de \Ze =& \Big[\emm Q\Ze+\epsilon^{-\frac{3}{4}}f\big(\epsilon^{\frac{1}{4}}\Ze,\epsilon^{\frac{1}{4}}\Ye\big)\Big]\de t + \sigma_Q\big(\epsilon^{\frac{1}{4}}\Ze,\epsilon^{\frac{1}{4}}\Ye\big)\de B,\\ \label{eq:timescaling2}
	\de \Ye =& \Big[\emm P\Ye+\epsilon^{-\frac{3}{4}}g\big(\epsilon^{\frac{1}{4}}\Ze,\epsilon^{\frac{1}{4}}\Ye\big)\Big]\de t + \sigma_P\big(\epsilon^{\frac{1}{4}}\Ze,\epsilon^{\frac{1}{4}}\Ye\big)\de B.
\end{align}
Let $\mathfrak{p}$ be the transformation defined in \eqref{diff:p}  and take $\Zpe$ such that 
\begin{equation} \label{ZY'}
	\mathfrak{p}\big(\euq\Zpe(t), \euq\Ype(t)\big)=\begin{pmatrix*}\euq\Zpe(t) + p\big(\euq\Zpe(t),\euq \Ype(t)\big) \\ \euq\Ype(t)  \end{pmatrix*} = \begin{pmatrix*} \euq \Ze(t) \\\euq \Ye(t) \end{pmatrix*}.
\end{equation}
As long as $\Zpe(t)$ is contained in a ball of radius smaller than a suitable constant by $\emq$, analogously to the deterministic case, the transformation $\mathfrak{p}$ can be used to remove the quadratic terms in the drift coefficient of equation \eqref{eq:timescaling1}. \\
Let $\mathrm{Id}$ be the identity matrix where the dimension depend on the context, then

\begin{align*}
	\mathrm{D}	\mathfrak{p}(\euq\Zpe,\euq\Ype) &=\euq\begin{pmatrix*}\mathrm{Id} +\mathrm{D}_z p(\euq\Zpe,\euq\Ype) &  \mathrm{D}_y p(\euq\Zpe,\euq\Ype) \\
		0 &  \mathrm{Id} \end{pmatrix*}\\
	& = \euq\big(\mathrm{Id} + A(\euq\Zpe,\euq\Ype)\big)
\end{align*}
where
\begin{equation*}
	A(z',y')=	\begin{pmatrix*} \mathrm{D}_z p(z',y') & \mathrm{D}_y p(z',y') \\
		0 & 0 \end{pmatrix*}
\end{equation*}

As long as $\Big(\euq\Zpe(t),\euq\Ype(t)\Big)$ remains close to the origin, define  $\mathfrak{p}^{-1}$ as the inverse function of $\mathfrak{p}$, then
\begin{align*}
	D(	\mathfrak{p})^{-1}\big(\euq\Ze,\euq\Ye\big) & =\big(D\mathfrak{p}\big(\euq\Zpe,\euq\Ype\big)\big)^{-1}\\
	&=\euq\Big(\mathrm{Id} - A \big(\euq\Zpe,\euq\Ype\big)+ \frac{1}{2} A^2\big(\euq\Zpe,\euq\Ype\big) \\
	&\qquad + R_A\big(\euq\Zpe,\euq\Ype\big)\Big),
\end{align*}
where $R_A$ is the remainder of the Taylor expansion around zero that, for fixed $(z',y')$, satisfies $$\lim_{\epsilon\to 0}\emm R_A^\epsilon(\euq z',\euq y')=0.$$

Here we have used the fact that $p(z',y')$ is a polynomial of degree two, so 
\[
A \big(\euq\Zpe,\euq\Ype\big) = O(\epsilon^{\frac14}).
\]

Then

\begin{align}\label{eq:newvariables1}
	&\begin{aligned}
		\de\Zpe =  \Big[& \emm Q\Zpe+f^{(3)}(\Zpe,\Ype)+\emq f^{(2)}(\Ype) + \etq r_f(\euq\Zpe,\euq\Ype)  \\ 
		& + \frac{1}{2}\eum r_{p,\sigma}(\euq\Zpe,\euq\Ype) \Big]\de t+ \sigma_Q'\big(\euq\Zpe,\euq\Ype\big)\de B,
	\end{aligned}\\ \label{eq:newvariables2}
	&\de \Ype = \Big[\emm P\Ype+\epsilon^{-\frac{3}{4}}g\big(\epsilon^{\frac{1}{4}}\Zpe,\epsilon^{\frac{1}{4}}\Ype\big)\Big]\de t + \sigma_P'\big(\epsilon^{\frac{1}{4}}\Zpe,\epsilon^{\frac{1}{4}}\Ype\big)\de B. 
\end{align} 
where 
\begin{enumerate}[i)]
	\item functions $f^{(3)}, f^{(2)}, r_f$ are defined in \eqref{eq:Carr1prime};
	\item $r_{p,\sigma}$ is a term coming from the quadratic variation in the Ito's formula; the function $r_{p,\sigma}( \, \cdot \, , \, \cdot \,)$ is locally bounded around $(0,0)$.
	\item $\sigma_Q'\big(z,y\big)=\Big(\mathrm{Id} - A(z,y) + \frac{1}{2} A^2(z,y) + R_A(z,y)\Big)\sigma_Q\big(\mathfrak{p}(z,y\big))$;
	\item $\sigma_{P}'\big(z,y\big) = \sigma_P\big(\mathfrak{p}(z,y\big))$.
\end{enumerate}

Let $h:\mathbb{R}^2\to\mathbb{R}^{n-2}$ the function that describes the centre manifold associated to the Hopf's bifurcation, then consider the 2-dimentional SDE 

\begin{equation} \label{eq:centre3}
	\tZe = \Big[\emm Q\tZe+f^{(3)}\big(\tZe,\emq h(\euq \tZe)\big) \Big]\de t + \sigma_{Q}\big(\euq\tZe,\emq h(\euq \tZe)\big)\de B. 
\end{equation}

Let $\mathfrak{q}$ be the transformation defined in \eqref{diff:q}  and take $\tZpe$ such that 
\begin{equation*}
	\mathfrak{q}\big(\euq\tZpe(t)\big) = \euq\tZpe(t) + q\big(\euq\tZpe(t)\big) = \euq\tZe(t).
\end{equation*}
As long as $\euq\tZpe(t)$ remains close to the origin, following the theory of normal form, it holds
\begin{equation} \label{eq:Z1Z2}
	\begin{aligned}
		\de \tZpe_1 &=  \Big[ -\emm\lambda_0\tZpe_2- \tZpe_1\big((\tZpe_1)^2+(\tZpe_2)^2\big)+\etq f^{(4)}_1\big(\euq\tZpe\big) \\
		&\quad  + \frac{1}{2}\eum r_{q,\sigma,1}(\euq\tZpe)  \Big]\de t +  \Big( \sigma_{Q}'' \big(\euq\tZe\big)\de B\Big)_1,\\
		\de \tZpe_2 &=  \Big[ +\emm\lambda_0\tZpe_1- \tZpe_2\big((\tZpe_1)^2+(\tZpe_2)^2\big)+\etq f^{(4)}_2\big(\euq\tZpe\big) \\
		&\quad + \frac{1}{2}\eum r_{q,\sigma,2}(\euq\tZpe)\Big]\de t + \Big(\sigma_{Q}'' \big(\euq\tZe\big)\de B\Big)_2,
	\end{aligned}
\end{equation}
where 
\begin{enumerate}[i)]
	\item function $f^{(4)}=\big(f^{(4)}_1,f^{(4)}_2\big)^{\mathsf{T}}$ is defined in \eqref{eq:z1z2normalform};
	\item $r_{q,\sigma}=\big(r_{q,\sigma,1},r_{q,\sigma,2}\big)^{\mathsf{T}}$ is a term coming from the quadratic variation in the Ito's formula;	the function $r_{q,\sigma}( \, \cdot \, , \, \cdot \,)$ is locally bounded around $(0,0)$.
	\item $\Big(\sigma_{Q}''\big(\euq\tZe\big)\de B\Big)_i$ is the $i$-th component of the diffusive part obtained as done in \eqref{eq:newvariables1}.
\end{enumerate}

We stress the fact that the above change of variables makes sense only if the processes $\big(\euq\Zpe(t),\euq\Ype(t)\big)_{t\in[0,T]}$ and $\big(\euq\tZpe(t)\big)_{t\in[0,T]}$ remain sufficiently close to the origin. In other words, equations \eqref{eq:Z1Z2} correspond to the original process in the new variables up to the time $\big(\Zpe(t),\Ype(t)\big)$ and $\tZpe(t)$ exit a ball centered in the origin and of radius proportional to $\emq$. When we will use equations \eqref{eq:Z1Z2}, the process will be stopped as it gets to a distance from the origin which is much smaller than $\emq$.

\section{Reduction principle}\label{sec:reduction}

In this section we show that the original $n$ dimensional system is approximated by a $2$ dimensional process that lives on the centre manifold, up to a suitable stopping time.

Define the following processes 
\begin{equation*}
	\pe(t):=\Zpe(t)-\tZpe(t),\quad \ue(t):=\Ype(t)-\emq h\big(\euq\Zpe(t)\big).
\end{equation*}
	
	Given $\Delta>0$, define the stopping time 
	\begin{equation}\label{def:stoppingZYZ}
		\tau_\Delta^\epsilon:=\inf\{t\in[0,T]: \max\{\|\big(\Zpe(t),\Ype(t)\big)^{\mathsf{T}}\|,\|\tZpe(t)\|\}>\Delta\}
	\end{equation} 
	and  the stopped processes:
	\begin{align*}
		\Zpe_\Delta(t) & :=\Zpe\big(t\wedge\tau_\Delta^\epsilon\big) ,\quad \Ype_\Delta(t):=\Ype\big(t\wedge\tau_\Delta^\epsilon\big),\quad \tZpe_\Delta(t) :=\tZpe\big(t\wedge\tau_\Delta^\epsilon\big), \\
		U^\epsilon_\Delta(t) & :=\ue\big(t\wedge\tau_\Delta^\epsilon\big),\quad \Phi^\epsilon_\Delta(t) :=\Phi^\epsilon\big(t\wedge\tau_\Delta^\epsilon\big).
	\end{align*}
	The main results of this section are the following theorems. 
	
	\begin{theorem}\label{teo:ued}
		Let $q\in(0,\frac{1}{4})$, $\beta<\frac{1}{2}$, and  $\Delta_\epsilon$ satisfying 
		$ \lim_{\epsilon \rightarrow 0} \epsilon^{\frac14} \Delta^3_\epsilon = 0 $. 
		Suppose that $\big(\Zpe(0),\Ype(0)\big)=\big(z^{'\epsilon}_0,y^{'\epsilon}_0\big)\to \big(z_0,y_0\big)$, as $\epsilon\to 0$; then, for any $T>0$, 
		\begin{equation*}
			\lim_{\epsilon\to 0}\mathsf{P}\left(\sup_{\epsilon^\beta\leq t \leq T} \|\ued(t)\|>\epsilon^q \right)= 0.
		\end{equation*}
	\end{theorem}

	This theorem shows that, in time interval $[\epsilon^{\beta},\tau_{\Delta_\epsilon}^\epsilon\wedge T]$, the process  $\Ype$ can be approximated by the process $\emq h\big(\euq\Zpe\big)$.
	The initial time $\epsilon^{\beta}$ is needed as the initial condition $(z_0,y_0)$ may be outside the center manifold, so the process needs a short time to get close to the center manifold.
	The proof is postponed to subsection \ref{subsection:ude}.

	\begin{theorem}\label{teo:ped}
		Take  $\gamma\in(0,\frac{1}{8})$ and $\Delta_\epsilon=\big[\ln(-\ln\epsilon)\big]^{1/2}$. Suppose that $\big(\Zpe(0),\Ype(0)\big)=\big(z^{'\epsilon}_0,y^{'\epsilon}_0\big)\to \big(z_0,y_0\big)$, as $\epsilon\to 0$, then, for any $T\geq0$
		\begin{equation*}
			\lim_{\epsilon\to 0}\mathsf{P}\left(\sup_{0\leq t\leq T}\|\ped(t)\|> \epsilon^\gamma\right)=0.
		\end{equation*}
	\end{theorem}

	This theorem shows that, in time interval $[0,\tau_{\Delta_\epsilon}^\epsilon\wedge T]$, the 2-dimensional process  $\Zpe$ can be approximated by the 2-dimensional process $\tZpe$. The proof is postpone to subsection \ref{subsection:pde}. \\
	
	With an abuse of notation in what follows the solutions to the equations \eqref{eq:newvariables1}, \eqref{eq:newvariables2} are called  $(\Ze,\Ye)$, the symbol $'$ is omitted also for $\tZpe$. Using Ito's formula the evolution of the processes $\ue$ is  described by:
	
	\begin{equation}\label{eq:uedNot}
		\begin{aligned}
			\de \ue = &  \Bigl[\emm P\ue  +\etq N_{\epsilon}\big(\euq\Ze,\euq\Ye\big) +\eum H\big(\euq\Ze,\euq\Ye\big) \Bigr] \de t \\
			& \quad +\sigma_{hP}\big(\euq\Ze,\euq\Ye\big)\de B
		\end{aligned}
	\end{equation}
	where 
	\begin{align} \label{newH}
		&\begin{aligned}
			N_{\epsilon}\big(z,\y\big) =  & g\big(z,y\big)-g\big(z, h\big(z\big)\big) \\
			&  +\etqp\textrm{D}h\big(z\big) \Big[F^\epsilon\big(z, h\big(z\big)\big)-F^\epsilon\big(z,\y\big)\Big],
		\end{aligned}\\
		& F^\epsilon(z,y) =\, \etq f^{(3)}(z,y)+\etq f^{(2)}(y)  
		+ \etq r_f(z,y) + \frac{1}{2}\eum r_{p,\sigma}(z,y),  \nonumber\\
		& H\big(z,y\big) = \frac{1}{2}\textrm{Tr}\Big(\sigma_P^{\mathsf{T}}\big(z,y\big)(D^2h)\big(z\big)\sigma_P\big(z,y\big)\Big), \nonumber\\
		& \sigma_{hP}\big(z,y\big) = \sigma_P\big(z,y) - \textrm{D}h\big(z\big)\sigma_Q\big(z,y\big).\nonumber
	\end{align}
	
	Applying Ito's formula the evolution of the process $\pe$ is described by 
	\begin{equation}\label{eq:pedNot}
		\begin{aligned}
			\de \pe  = & \Bigl[\emm Q\pe +\emq f^{(2)}(\Ye) +  G_\epsilon\big(\Ze,\Ye,\tZe\big) + R_\epsilon(\euq\Ze,\euq\Ye)  \Bigr]\de t \\
			& +\sigma_{hQ}\big(\euq\Ze,\euq\Ye,\euq\tZe\big)\de B,
		\end{aligned}
	\end{equation}
	where 
	
	\begin{align*} 
		&G_\epsilon\big(z,y,\tilde{z} \big)  =  f^{(3)}\big(z,y\big)-f^{(3)}\big(\tilde{z},\emq h\big(\euq \tilde{z}\big)\big), \\
		&R_\epsilon(z,y)=\etq r_f(z,y)+ \frac{1}{2}\eum r_{p,\sigma}(z,y),  \\ 
		&\sigma_{hQ}\big(z,y,\tilde{z}\big) = \sigma_Q\big(z,y\big) - \sigma_Q\big(\tilde{z},h\big(\tilde{z}\big)\big).
	\end{align*}
	
	Using the definition of $\ue$, $\pe$, and $\tZe$ the equations \eqref{eq:uedNot} and \eqref{eq:pedNot} can be written as 
	\begin{align} \label{eq:uedZ}
		&\begin{aligned}
			\de \ue & = \Bigg[\emm P\ue + \etq N\Big(\euq(\pe+\tZe),\euq\ue+h\big(\euq(\pe+\tZe)\big)\Big)  \\ 
			&\qquad +\frac{1}{2}\eum H\Big(\euq(\pe+\tZe),\euq\ue+h\big(\euq(\pe+\tZe)\big)\Big) \Bigg]\de t \\ 
			&\quad +\sigma_{hP}\Big(\euq(\pe+\tZe),\euq\ue+h\big(\euq(\pe+\tZe)\big)\Big)\de B;
		\end{aligned} \\ \label{eq:pedZ}
		&\begin{aligned}
			\de \pe & = \Bigg[\emm Q\pe +\emq f^{(2)}\Big(\ue+\emq h\big(\euq(\pe+\tZe)\big)\Big)\\ 
			&\qquad + G_\epsilon\Big(\pe+\tZe,\ue+\emq h\big(\euq(\pe+\tZe)\big),\tZe\Big) \\ 
			&\qquad + R_\epsilon\Big(\euq(\pe+\tZe),\euq\ue+h\big(\euq(\pe+\tZe)\big)\Big) \Bigg]\de t \\ 
			&\quad + \sigma_{hQ}\Big(\euq(\pe+\tZe),\euq\ue+h\big(\euq(\pe+\tZe)\big),\euq\tZe\Big)\de B.  
		\end{aligned}
	\end{align}

	
	\subsection{Proof Theorem \ref{teo:ued} }\label{subsection:ude}

	Set $u_0^\epsilon=y^{\epsilon}_0-\emq h(\euq z^{\epsilon}_0)$. By applying the variation of constants to equation \eqref{eq:uedZ} for the stopped process $\ued$, for any $t\in[0,T\wedge\tau_{\Delta_\epsilon}^\epsilon]$ we obtain:
	\begin{equation}\label{eq:varcostU}
		\begin{aligned}
			\ued(t) = & e^{\emm Pt}u_0^\epsilon	\\
			&+\int_0^{t \wedge \tau_{\Delta_\epsilon}^\epsilon}  e^{\emm P(t-r)} \Bigg[\etq N\Big(\euq(\ped+\tZed),\euq\ued+h\big(\euq(\ped+\tZed)\big)\Big)  \\ 
			&\qquad	+ \eum H\Big(\euq(\ped+\tZed),\euq\ued+h\big(\euq(\ped+\tZed)\big)\Big)\Bigg]\de r \\  
			&+\int_0^{t \wedge \tau_{\Delta_\epsilon}^\epsilon} e^{\emm P(t-r)}\sigma_{hP}\Bigl(\euq(\ped+\tZed),\euq\ued+h\big(\euq(\ped+\tZed)\big) \Bigr)\de B(r).
		\end{aligned}
	\end{equation}
	For convenience of notation, the time is sometimes omitted. 
	
	We estimate separately the three terms in the r.h.s. of \eqref{eq:varcostU} produced by the variation of constants.
	
	\textit{STEP 1:} The bound
	
	\[
	\left\|\int_0^{t \wedge \tau_{\Delta_\epsilon}^\epsilon}   e^{\emm P(t-r)}\eum H\Big(\euq(\ped+\tZed),\euq\ued+h\big(\euq(\ped+\tZed)\big)\Big)\de r \right\|  \leq C T \epsilon^{\frac12},
	\]
	for some $C>0$, follows readily from the fact that the function $H$ is locally bounded. \\

	\textit{STEP 2:} We now estimate the contribution of the martingale term in \eqref{eq:varcostU}.
	Define, for any $0\leq t'<t\leq T\wedge\tau_{\Delta_\epsilon}^\epsilon$, 
	\begin{equation*}
		M_u^\epsilon(t',t)=\int_0^{t'} e^{\emm P(t-r)}\sigma_{hP}(r)\de B(r).
	\end{equation*}
	where 
	\[
	\sigma_{hP}(r)=\sigma_{hP}\Bigl(\euq(\ped(r)+\tZed(r)),\euq\ued(r)+h\big(\euq(\ped(r)+\tZed(r))\big) \Bigr).
	\]
	We write  $M_u^\epsilon(t)$ for $M_u^\epsilon(t,t)$. The reason for introducing the time $t'$ will be clear later: in moment estimates we will use the Ito's rule for $M_u^\epsilon(t',t)$ as function on $t'$ only, thus avoiding differentiation of the exponential $e^{\emm P(t-r)}$.
	
	The dependence of $\sigma_{hP}$ on the variables $\ued,\ped$, and $\tZed$ is omitted to make the following calculations easier to read.\\
	Given $s',s>0$ such that $s'< s\leq t $ and $s'\leq t'$ 
	\begin{equation}\label{differenceMu}
		\begin{aligned}
			M_u^\epsilon(t',t)-M_u^\epsilon(s',s)  = & \int_{s'}^{t'} e^{\emm P(t-r)}\sigma_{hP}(r)\de B(r) \\ 
			& +  \int_{0}^{s'} \left( e^{\emm P(t-r)} - e^{\emm P(s-r)}\right) \sigma_{hP}(r) \de B(r).
		\end{aligned}
	\end{equation}
	The second moment of \eqref{differenceMu} is bounded by
	\begin{equation*}
		\Expect	\left[\|M_u^\epsilon(t',t)-M_u^\epsilon(s',s)\|^2\right]\leq \Expect[I^\epsilon(t,s',t')]+\Expect[\tilde{I}^\epsilon(t,s,s')],
	\end{equation*}
	where 
	\begin{align*}
		\Expect[I^\epsilon(t,s',t')] & = 2 \int_{s'}^{t'} \Expect\Bigg[\mathrm{Tr}\left(e^{\emm P(t-r)}\sigma_{hP}(r) \sigma_{hP}^{\mathsf{T}}(r)\big(e^{\emm P(t-r)} \big)^{\mathsf{T}}\right)\Bigg]\de r,\\
		\Expect[ \tilde{I}^\epsilon(t,s,s')] & = 2 \int_{0}^{s'} \Expect\Bigg[\mathrm{Tr}\left(\big(e^{\emm P(t-r)}- e^{\emm P(s-r)}\big)\sigma_{hP} \sigma_{hP}^{\mathsf{T}}\big(e^{\emm P(t-r)} - e^{\emm P(s-r)} \big)^{\mathsf{T}}\right)\Bigg]\de r.
	\end{align*}
	Since the eigenvalues of $P$ have negative real parts, there exist positive constants $\lambda^*$ and $C^*$ such that for any $r>0$ and $y\in\mathbb{R}^{n-2}$, $\|e^{rP}y\|\leq C^*e^{-\lambda^*r}\|y\|$, see \cite{carr2012applications} equation (2.3.6). Thus
	\begin{align}\nonumber
		\Expect[I^\epsilon(t,s',t')]& \leq 2(n-2)C^*\int_{s'}^{t'}\|\sigma_{hP}\|_{\infty,1}^2 e^{-2\lambda^*\emm (t-r)}\de r\\ \nonumber
		&\quad = \underbrace{\frac{(n-2)C^*\|\sigma_{hP}\|_{\infty,1}^2}{\lambda^*}}_{K_I}\eum \left(e^{-2\lambda^*\emm (t-t')}-e^{-2\lambda^*\emm (t-s')}\right)\\ \nonumber
		&\leq K_I \eum \left(1-e^{-2\lambda^*\emm (t-s')}\right)\\ \label{ineq:ExI}
		&\leq K_I \eum \left(\emm (t-s')\mathbbm{1}_{\emm (t-s')\leq \epsilon^\alpha}+\mathbbm{1}_{\emm (t-s')> \epsilon^{\alpha}} \right),
	\end{align}
	where $\alpha>0$ and  $\|\cdot\|_{\infty,1}$ is the sup-norm restricted to the ball of radius $1$. Here we have used the fact that $\sigma_{hP}$ is locally bounded, and evaluated in processes that are stopped in a ball of radius $\Delta_\epsilon$
	\\
	To estimate $\Expect[ \tilde{I}^\epsilon(t,s,s')] $, for $y\in\mathbb{R}^{n-2}$ write 
	$$\left(e^{\emm P(t-r)}- e^{\emm P(s-r)}\right)y=\left(e^{\emm P(s-r)}\big(e^{\emm P(t-s)}-\Id\big)\right)y.$$ 
	Take $\alpha>0$ 
	and consider two cases: either $\emm (t-s')\leq \epsilon^\alpha$  or $\emm (t-s')> \epsilon^\alpha$
	\begin{align*}
		{\bf 1)}\quad \Bigg(e^{\emm P(s-r)}&\Big(e^{\emm P(t-s)}-\Id\Big)\Bigg)y\mathbbm{1}_{\emm (t-s')\leq \epsilon^\alpha}\\
		&=e^{\emm P(s-r)}\big(-\emm(t-s)P + r_P\big(\emm (t-s)\big)\Id\big)y\mathbbm{1}_{\emm (t-s')\leq \epsilon^\alpha}
	\end{align*}
	where $\|r_P(\emm (t-s))\|=o\big(\emm(t-s)\big)$;
	\begin{equation*}
		{\bf 2)} \quad \left\|\left(e^{\emm P(s-r)}\big(e^{\emm P(t-s)}-\Id\big)\right)y\mathbbm{1}_{\emm (t-s')> \epsilon^\alpha}\right\|\leq K_1 \|y\|\mathbbm{1}_{\emm (t-s')> \epsilon^\alpha}
	\end{equation*}
	where $K_1$ is positive constant independent on $\epsilon$.\\ 
	It follows that the operator norm of $e^{\emm P(s-r)}\big(e^{\emm P(t-s)}-\Id\big)$ can be bounded:
	\begin{equation}\label{ineq:operatornorm}
		\begin{aligned}
			\big\|e^{\emm P(s-r)}&\big(e^{\emm P(t-s)}-\Id\big)\big\|  \leq e^{-\emm \lambda^*(s-r)}\Big( K_1\mathbbm{1}_{\emm (t-s') > \epsilon^\alpha}\\ 
			& + K_2\Big(\emm(t-s) +o\big(\emm(t-s)\big)\Big )\mathbbm{1}_{\emm (t-s')\leq \epsilon^\alpha}   \Big),
		\end{aligned}
	\end{equation}
	where $K_2$ is a positive constant independent on $\epsilon$. Using \eqref{ineq:operatornorm} the expected value $\Expect[ \tilde{I}^\epsilon(t,s,s')]$  can be bounded as follows:
	
	\begin{align} \nonumber
		\Expect[ \tilde{I}^\epsilon(t,s,s')]  & \leq 2(n-2)C^*\|\sigma_{hP}\|_{\infty,1}^2\int_{0}^{s'}e^{-2\lambda^*\emm (s-r)} \de r \Big(  K_1\mathbbm{1}_{\emm (t-s')> \epsilon^\alpha} \\ \nonumber
		& \quad +K_2\big(\emm(t-s)+o\big(\emm(t-s)\big) \big)\mathbbm{1}_{\emm (t-s')\leq \epsilon^\alpha} \Big)^2\\ \nonumber
		&	\leq C_1 \eum \Big(K_1^2\mathbbm{1}_{\emm (t-s')> \epsilon^\alpha} \\ \nonumber
		& \quad +K_2^2\big(\emm(t-s)+o\big(\emm(t-s)\big)\big)^2\mathbbm{1}_{\emm (t-s')\leq \epsilon^\alpha} \Big)\\ \label{ineq:ExII}
		&	\begin{aligned} \leq &\, \tilde{K}_1\eum \mathbbm{1}_{\emm (t-s')> \epsilon^\alpha} \\
			&+\tilde{K}_2\Big(\emm(t-s)^2+o\big(\emm(t-s)^2\big)\Big)\mathbbm{1}_{\emm (t-s')\leq \epsilon^\alpha},
		\end{aligned}
	\end{align}
	
	where $C_1,\tilde{K}_1$ and $\tilde{K}_2$ are positive constants independent on $\epsilon$. Inequalities \eqref{ineq:ExI} and \eqref{ineq:ExII} give 
	\begin{multline*} 
		\Expect	\big[\|M_u^\epsilon(t',t)-M_u^\epsilon(s',s)\|^2\big]\leq C \left[ \eum \mathbbm{1}_{\emm (t-s')> \epsilon^\alpha} \right.\\
		\left. + \Big( (t-s')+\emm(t-s)^2+o\big(\emm(t-s)^2)\big)\Big)\mathbbm{1}_{\emm (t-s')\leq \epsilon^\alpha}\right].
	\end{multline*}
	
	We now need to extend the above inequality to higher moments:
	
	\begin{equation}\label{ineq:functionl}
		\begin{aligned}
			\Expect	\big[\|&M_u^\epsilon(t',t)-M_u^\epsilon(s',s)\|^{2m}\big]\leq C_m \Big[ \epsilon^{m/2}\mathbbm{1}_{\emm (t-s')> \epsilon^\alpha} \\
			&+ \Big( (t-s')^m+\epsilon^{-m/2}(t-s)^{2m}+o\big(\epsilon^{-m/2}(t-s)^{2m})\big)\Big)\mathbbm{1}_{\emm (t-s')\leq \epsilon^\alpha}\Big].
		\end{aligned}
	\end{equation}
	
	We now show inequality \eqref{ineq:functionl} for $m=2$. The same argument allows to extend inductively the estimates for larger $m$. Using \eqref{differenceMu}:
	\begin{equation} \label{m12}
		\begin{aligned}
			\Expect	\big[\|M_u^\epsilon(t',t)-M_u^\epsilon(s',s)\|^{4}\big]\leq &4  \Expect	\big[\|\int_{s'}^{t'} e^{\emm P(t-r)}\sigma_{hP}(r)\de B(r)\|^4\big] \\ 
			& +4 \Expect	\big[\| \int_{0}^{s'} \left( e^{\emm P(t-r)} - e^{\emm P(s-r)}\right) \sigma_{hP}(r) \de B(r)\|^4\big].
		\end{aligned}
	\end{equation}
	
	Call $\tilde{M}_{t'}=\int_{s'}^{t'} e^{\emm P(t-r)}\sigma_{hP}(r)\de B(r)$ then using Ito's formula (here is where the time $t'$ becomes convenient)
	\begin{equation*}
		\begin{split}
			\|\tilde{M}_{t'}\|^4= & \int_{s'}^{t'}4\|\tilde{M}_{r}\|^2\tilde{M}_{r}^{\mathsf{T}}\de \tilde{M}_{r}+ 4\int_{s'}^{t'} \left\|\tilde{M}_{r}^{\mathsf{T}}e^{\emm P(t-r)}\sigma_{hP}(r)  \right\|^2 \de r \\ & + 2\int_{s'}^{t'}\|\tilde{M}_{r}\|^2\mathrm{Tr}\Big(e^{\emm P(t-r)}\sigma_{hP}(r) \sigma_{hP}^{\mathsf{T}}(r)\big(e^{\emm P(t-r)} \big)^{\mathsf{T}}\Big)\de r.
		\end{split}
	\end{equation*}
	
	Taking the expected value and noting that 
	\[
	\Expect\big[\|\tilde{M}_{r}\|^2\big] = \Expect[I^\epsilon(t,s',r)]
	\]
	estimated above:
	\begin{align*}
		\Expect	\big[\|\tilde{M}_{t'}\|^4\big]= & 4\int_{s'}^{t'} \Expect	\left[\left\|\tilde{M}_{r}^{\mathsf{T}}e^{\emm P(t-r)}\sigma_{hP}(r) \right\|^2 \right] \de r \\
		+ &2\int_{s'}^{t'}\Expect\Big[\|\tilde{M}_{r}\|^2\mathrm{Tr}\Big(e^{\emm P(t-r)}\sigma_{hP}(r) \sigma_{hP}^{\mathsf{T}}(r)\big(e^{\emm P(t-r)} \big)^{\mathsf{T}}\Big)\Big]\de r\\
		\leq&C_1 \int_{s'}^{t'}e^{-2\lambda^*\emm (t-r)}\Expect\big[\|\tilde{M}_{r}\|^2\big]\de r\\
		\leq C_2 \eum &\left(\emm (t-s')\mathbbm{1}_{\emm (t-s')\leq \epsilon^\alpha}+\mathbbm{1}_{\emm (t-s')> \epsilon^{\alpha}} \right)\int_{s'}^{t'}e^{-2\lambda^*\emm (t-r)}\de r		\\
		\leq&C_3 \epsilon \left(\emm (t-s')\mathbbm{1}_{\emm (t-s')\leq \epsilon^\alpha}+\mathbbm{1}_{\emm (t-s')> \epsilon^{\alpha}} \right)^2\\
		\leq& C_4 \epsilon \left(\epsilon^{-1} (t-s')^2\mathbbm{1}_{\emm (t-s')\leq \epsilon^\alpha}+\mathbbm{1}_{\emm (t-s')> \epsilon^{\alpha}} \right).
	\end{align*}
	for suitable constants $C_1,C_2,C_3,C_4$. The second term in the r.h.s. of \eqref{m12} is estimated with the same technique, using \eqref{ineq:ExII}.

	By Fatou's lemma:
	\begin{equation}\label{estMu}
		\Expect	\left[\|M_u^\epsilon(t)-M_u^\epsilon(s)\|^{2m}\right]\leq \liminf_{\substack{s'\uparrow s\\t'\uparrow t}}\Expect	\left[\|M_u^\epsilon(t',t)-M_u^\epsilon(s',s)\|^{2m}\right] \leq C_m l_\epsilon^m(t-s), 
	\end{equation}
	
	where
	\[
	l_\epsilon^m(t-s) = \Big[ \epsilon^{m/2}\mathbbm{1}_{\emm (t-s)> \epsilon^\alpha} 
	+ \Big( (t-s)^m+\epsilon^{-m/2}(t-s)^{2m}+o\big(\epsilon^{-m/2}(t-s)^{2m})\big)\Big)\mathbbm{1}_{\emm (t-s)\leq \epsilon^\alpha}\Big].
	\]
	
	We now apply  the Garsia, Rademich and Rumsey Lemma in \cite{stroock1997multidimensional}, which is reported below. 
	\begin{lemma}\label{lemma:GRR}
		Let $p$ and $\Psi$ be continuous, strictly increasing functions on $(0, \infty)$ such that $p(0)=\Psi(0)=0$ and $\lim_{t\to\infty} \Psi(t)=\infty$. Given $T>0$ and $\phi$ continuous on $(0, T)$ and taking its values in a Banach space $(E,\|\cdot\|)$, if
		\begin{equation*}
			\int_{0}^T\int_0^T\Psi\left(\frac{\|\phi(t)-\phi(s)\|}{p(|t-s|)}\right)\de s\de t \leq  W <\infty.
		\end{equation*}
		then for $0 \leqslant s \leqslant t \leqslant T$ :
		\begin{equation*}
			\|\phi(t)-\phi(s)\| \leqslant 8 \int_0^{t-s} \Psi^{-1}\left(\frac{4 W}{u^2}\right) p(\de u) .
		\end{equation*}
	\end{lemma}
	Lemma \ref{lemma:GRR} is applied taking
	\begin{equation*}
		\phi(t)=M_u^\epsilon(t),\ p(x)=x^{\frac{2+\zeta}{2m}},\ \Psi(x)=x^{2m},
	\end{equation*}
	and $0<\zeta<<1$. Then there is a constant $C$ such that
	\begin{equation*}
		\|M_u^\epsilon(t)-M_u^\epsilon(s)\|^{2m}\leq C (t-s)^{\zeta}W,
	\end{equation*}
	for every $s$ and $t$ and 
	
	\[
	W = \int_0^T \int_0^T \frac{\|M^{\epsilon}_u(t) - M^{\epsilon}_u(s)\|^{2m}}{|t-s|^{2+\zeta}} \de s \de t.
	\]
	By \eqref{estMu}
	\begin{equation}\label{ineq:ExpectB}
		\Expect[W]\leq \int_0^{T}\int_0^{T}\frac{C_m l_\epsilon^m(t-s)}{|t-s|^{2+\zeta}}\de s\de t.
	\end{equation}
	The right hand side of \eqref{ineq:ExpectB} can be written as 
	\begin{equation*}
		\int_0^{T}\int_0^{T}\frac{C_m l_\epsilon^m(t-s)}{|t-s|^{2+\zeta}}\de s\de t = I_l+\tilde{I}_l
	\end{equation*}
	where 
	\begin{equation*}
		I_l = C_m \int_0^{T}\int_0^{T} \frac{\epsilon^{m/2} }{|t-s|^{2+\zeta}}\mathbbm{1}_{\emm (t-s')> \epsilon^\alpha}\de s\de t\leq C_{m,\zeta} \epsilon^{m/2-(\alpha+1/2)(1+\zeta)},
	\end{equation*}
	and 
	\begin{align*}
		\tilde{I}_l & =  C_m \int_0^{T}\int_0^{T} \frac{|t-s|^{m/2}+\epsilon^{-m/2}|t-s|^{m/2} + o(\epsilon^{-m/2}|t-s|^{m/2})  }{|t-s|^{2+\zeta}}\mathbbm{1}_{\emm (t-s')\leq \epsilon^\alpha}\de s\de t\\
		& \leq \tilde{C}_{m,\zeta}\left(\epsilon^{(1/2+\alpha)(m/2-1-\zeta)}+\epsilon^{(1/2+\alpha)(2m-1-\zeta)-m/2} \right),
	\end{align*}
	
	for suitable constants $ C_{m,\zeta}$ and  $\tilde{C}_{m,\zeta}$ depending on $m,\zeta$ and $T$.\\
	Note that the expected valued  $\Expect[W]$ is small with $\epsilon$ if  $m>(1+\zeta)\max\{2,2(\alpha+\frac{1}{2}),\frac{2\alpha+1}{4\alpha+1}\}$. Now, applying Garsia-Rodemich-Rumsey Lemma we have
	\begin{equation}\label{ineq:ExpectM}
		\Expect \left[\sup_{0\leq s\leq t\leq T}\frac{\|M_u^\epsilon(t)-M_u^\epsilon(s)\|^{2m}}{|t-s|^\zeta}\right]\leq C e^{m/2-(\alpha+1/2)(1+\zeta)}
	\end{equation}
	
	which implies
	\begin{equation}\label{ineq:ProbM}
		\mathsf{P}\left(\sup_{0\leq t\leq T}\|M_u^\epsilon(t)\|\geq \frac12 \epsilon^q \right) \leq C T^\zeta\epsilon^{m/2-(\alpha+1/2)(1+\zeta)-2mq}.
	\end{equation}
	for any $q\in(0,\frac{1}{4})$
	for a possibly different constant $C>0$.
	\\
	\textit{STEP 3:} 
	We now complete the estimate of \eqref{eq:varcostU}.
	Suppose that $m>\frac{2}{1-4q}(\alpha+\frac{1}{2})(1+\zeta)$ and define the event 
	$$A_u^\epsilon=\{\sup_{0\leq t\leq T}\|M_u^\epsilon(t)\|\geq \frac12 \epsilon^q\}.$$
	Applying \eqref{ineq:ProbM}, as $\epsilon\to0$, the probability of the event $A_u^\epsilon$ goes to zero, i.e. $\mathsf{P}(A_u^\epsilon)\to 0$.\\
	In the complementary event $(A_u^\epsilon)^c$, for any $t\in[0,T\wedge\tau_{\Delta_\epsilon}^\epsilon]$, using equation \eqref{eq:varcostU} it holds
	
	\begin{equation}\label{eq:estimate u A}
		\begin{aligned}
			\|\ued(t)\|\leq &  \, e^{-\lambda^*\emm t}\|u_0^\epsilon\| \\ & +\int_{0}^te^{-\lambda^*\emm (t-r)}\Big(\etq \|N\big(\euq(\pe+\tZe),\euq\ue+h\big(\euq(\pe+\tZe)\big)\big)\|\ \\
			& + K\Delta_\epsilon^2\eum\Big)\de r + \frac12 \epsilon ^q.
		\end{aligned}
	\end{equation}
	
	It is relevant to note that \eqref{eq:estimate u A} continues to hold also for $t>\tau_{\Delta_\epsilon}^\epsilon$: indeed the r.h.s. of \eqref{eq:estimate u A} remains constant for $t>\tau_{\Delta_\epsilon}^\epsilon$, while the l.h.s. increases.
	Define $\tau_q^\epsilon=\inf\{t > \epsilon^\beta: ||\ued(t)||\geq \epsilon^q\}$. 
	To complete the proof it is enough to show that $\tau_q^\epsilon>T\wedge\tau_{\Delta_\epsilon}^\epsilon$ almost surely.
	We observe that  the function $N$ in \eqref{eq:estimate u A} can be bounded as follow, using \eqref{newH},  and taking into account that the time variable, that we omit, is $r \leq T\wedge\tau_{\Delta_\epsilon}^\epsilon$:

	\begin{align*}
		\|N&\Big(\euq(\pe+\tZe),\euq\ue+h\big(\euq(\pe+\tZe)\big)\Big)\|\\
		&\leq  \|\underbrace{g\Big(\euq(\pe+\tZe),\euq\ue+h\big(\euq(\pe+\tZe)\big)\Big)-g\Big(\euq(\pe+\tZe),h\big(\euq(\pe+\tZe)\big)\Big)}_{a)}\|\\
		& +\etqp \underbrace{\|\mathrm{D}h\Big(\euq(\pe+\tZe)\Big)\|}_{b)}\\
		&\cdot \underbrace{\|F^\epsilon\Big(\euq(\pe+\tZe),h\big(\euq(\pe+\tZe)\big)\Big)-F^\epsilon\Big(\euq(\pe+\tZe),\euq\ue+h\big(\euq(\pe+\tZe)\big)\Big)\|}_{c)}.
	\end{align*}
	Recall that $g(0,0)=0$, $Dg(0,0)=\underbar{0}$ and $Dh(0,0)=\underbar{0}$ then
	\begin{itemize}
		\item $a)\leq K_g \eum \Delta_\epsilon \|\ued\|$,
		\item $b)\leq  K_h \euq \Delta_\epsilon$,
	\end{itemize}
	where the constants $K_g,K_h$ does not depend on $\epsilon$.
	\begin{align*}
		c) \leq & \Big\| f^{(3)}\Big(\pe+\tZe,\emq h\Big(\euq(\pe+\tZe)\Big)\Big) - f^{(3)}\Big(\pe+\tZe,\ue+\emq h\big(\euq(\pe+\tZe)\big)\Big)  \Big\|\\
		& + \emq \Big\| f^{(2)}\Big(\emq h\Big(\euq(\pe+\tZe)\Big)\Big) - f^{(2)}\Big(\pe+\tZe,\ue+\emq h\big(\euq(\pe+\tZe)\big)\Big)  \Big\| \\
		& + \Big\| R_\epsilon\Big(\euq\big(\pe+\tZe\big), h\Big(\euq\big(\pe+\tZe\big)\Big)\Big) - R_\epsilon\Big(\euq\big(\pe+\tZe\big),\euq\ue+h\big(\euq(\pe+\tZe)\big)\Big)  \Big\|\\
		\leq &  K_f\Delta_\epsilon^2\|\ued\| + \emq K_{f^{(2)}}\Delta_\epsilon \|\ued\| + \euq K_{R}\Delta_\epsilon ,
	\end{align*}
	where the constants $K_{f},  K_{f^{(2)}}$ and  $K_{R}$ do not depend on $\epsilon$ and 
	\[
	R_\epsilon(\cdot)=\etq r_f(\cdot) + \frac{1}{2}\eum r_{p,\sigma}(\cdot). 
	\]
	Take a constant $\tilde{K}$ larger than all the constants $K_{(\cdot)}$ and  $t \in [\epsilon^{\beta}, T\wedge\tau_{\Delta_\epsilon}^\epsilon\wedge \tau_q]$. We have
	
	\begin{equation}
		\begin{split}
			\|\ued(t)\| & \leq e^{-\lambda^*\emm t}\|u_0^\epsilon\|+\tilde{K}\int_{0}^{t} e^{-\lambda^*\emm (t-r)}\emq\Delta_\epsilon^3 \|\ued\| \de r + \frac12 \epsilon ^q \\ 
			& \leq e^{-\lambda^*\emm t}\|u_0^\epsilon\|+\tilde{K} \int_{0}^{t} e^{-\lambda^*\emm (t-r)}\emq\Delta_\epsilon^3 2\epsilon^q \de r + \frac12 \epsilon ^q \\
			& \leq e^{-\lambda^*\emm \epsilon^\beta}\|u_0^\epsilon\| + \frac{2\tilde{K}}{\lambda^*}\Delta_\epsilon^3 \epsilon^{\frac{1}{4}+q} + \frac12 \epsilon ^q.
		\end{split}
	\end{equation}
	
	Observe that
	\[
	\lim_{\epsilon \rightarrow 0} \epsilon^{-q} \left[e^{-\lambda^*\emm \epsilon^\beta}\|u_0^\epsilon\| + \frac{2\tilde{K}}{\lambda^*}\Delta_\epsilon^3 \epsilon^{\frac{1}{4}+q}\right] = 0.
	\]
	This implies that $\|\ued(T\wedge\tau_{\Delta_\epsilon}^\epsilon\wedge \tau_q)\| \ll \epsilon^q$, 
	and therefore $\tau_q^\epsilon > T\wedge \tau^\epsilon_{\Delta_\epsilon}$, which completes the proof.
	
	
	\subsection{Proof Theorem \ref{teo:ped} } \label{subsection:pde}


The proof of this theorem retraces the steps used in the proof of Theorem \ref{teo:ued} applied to the process $\nod(t)=\|\ped(t)\|^2$, then some calculation is omitted. Applying Ito's formula, for any $t\in[0,T\wedge\tau_{\Delta_\epsilon}^\epsilon]$, using the fact that $Q + Q^{\mathsf{T}} = 0$, we have

\begin{equation}\label{eq:evolutionN}
	\begin{aligned}
		\de \nod(t) & = \Bigg\{2\emq(\ped)^{\mathsf{T}} f^{(2)}\Big(\ued+\emq h\big(\euq(\ped+\tZed)\big)\Big)\\ 
		&\qquad + 2(\ped)^{\mathsf{T}}  G_\epsilon\Big(\ped+\tZed,\ued+\emq h\big(\euq(\ped+\tZed)\big),\tZed\Big) \\ 
		&\qquad + 2(\ped)^{\mathsf{T}}  R_\epsilon\Big(\euq(\ped+\tZed),\euq\ued+h\big(\euq(\ped+\tZed)\big)\Big) \\ 
		&\qquad + \frac{1}{2}\mathrm{Tr}\Big(\sigma_{hQ}^{\mathsf{T}}\sigma_{hQ}\Big) \Bigg\}\de t \\ 
		&\quad + 2(\ped)^{\mathsf{T}} \sigma_{hQ}\de B(t)  
	\end{aligned}
\end{equation}
where $\sigma_{hQ}(t)=\sigma_{hQ}\Big(\euq(\ped(t)+\tZed(t)),\euq\ued+h\big(\euq(\ped(t)+\tZed(t))\big),\euq\tZed(t)\Big)$.\\
For convenience of notation, the time is sometimes omitted. The proof is divided in two steps.\\

\textit{STEP 1: } Define, for any $0\leq t\leq T\wedge\tau_{\Delta_\epsilon}^\epsilon$,  
\begin{align}\nonumber
	M_{n}^\epsilon(t)= & \int_{0}^{t}2(\ped)^{\mathsf{T}}  \sigma_{hQ}(r)\de B(r).
\end{align}
Let $A$ be a matrix-valued function and $v$ be a vector-valued function, then, for any $i,j$, denote by $\big(A\big)_{ij}(\cdot)$ the $ij$th-element of $A$ and  by $\big(v\big)_{i}(\cdot)$  the $i$th-element of $v$. Observe that
\begin{equation*}
	\begin{aligned}		
		\Big|\Big(\sigma_{hQ}\Big)_{ij}\Big(\euq\big(\ped+\tZed\big),& \euq\ued+h\big(\euq\big(\ped+\tZed\big)\big)\Big)\Big| \\
		&\leq   \sum_k\Big| \big(\ped\big)_k \partial_{z_k}\big(\sigma_{Q}\big)_{ij}\big(\tZed,h(\euq\tZed\big)\big)\Big| \\
		&\quad +\sum_l\Big| \big(\ued\big)_l \partial_{y_l}\big(\sigma_{Q}\big)_{ij}\big(\tZed,h(\euq\tZed\big)\big)\Big| + \big(\tilde{R}\big)_{ij}\big(\ped,\ued\big),
	\end{aligned}
\end{equation*}
where $\lim_{\epsilon\to 0}\frac{\| \big(\tilde{R}\big)_{ij}\big(\ped,\ued\big) \|}{\euq\Delta_\epsilon} =0$. Then it holds
\begin{align}
	\Big|\Big(\sigma_{hQ}\Big)_{ij}&\Big(\euq\big(\ped+\tZed\big), \euq\ued+h\big(\euq\big(\ped+\tZed\big)\big)\Big)\Big|  
	\leq K_\sigma\Big(\euq\|\ped\|+\euq\|\ued\|\Big) \label{ineq:sigmahQ}
\end{align}
for a suitable constant $K_\sigma$ that does not depend on $\epsilon$. The quadratic variation of $M_{n}^\epsilon$ can be estimate as follows 
\begin{align*}
	\Expect \left[	|M_{n}^\epsilon(t)|^2\right] = & \int_{0}^{t}\Expect\left[\mathrm{Tr}\left(\sigma_{hQ}(r)\big(\sigma_{hQ}(r)\big)^\mathsf{T} \right) \right]\de r\\
	& \leq C\Big\{\eum\Expect\big[\|\ped(r)\|^2\big]+\eum\Expect\big[\|\ued(r)\|^2\big]+\eum\Expect\big[\|\ped(r)\|\|\ued(r)\|\big]\Big\}	\\
	& \leq C' \eum \Delta_\epsilon^2,
\end{align*}
where $C, C'$ are suitable  constants that do not depend on $\epsilon$ and  they are obtained using inequality \eqref{ineq:sigmahQ} . 
By the Burkholder-Davis-Gundy inequality this implies
\begin{equation*}
	\Expect\left[\sup_{0\leq t \leq T} |M_n^\epsilon(t)|^2\right]\leq \tilde{C}\eum \Delta_\epsilon^2,
\end{equation*}
and, taking $\alpha\in(0,\frac{1}{4})$ and using Markov inequality, 
\begin{equation*}
	\mathsf{P}\left(\sup_{0\leq t \leq T} |M_n^\epsilon(t)|>  \epsilon^\alpha \right) \leq C \epsilon^{2(1/4-\alpha)}\Delta_\epsilon^2.
\end{equation*}

\textit{STEP 2: } Define the event $A_n^\epsilon=\{\sup_{0\leq t \leq T} |M_n^\epsilon(t)|>\epsilon^\alpha \}$ and  observe that $P(A_n^\epsilon)$ converges to zero, as $\epsilon$ goes to zero. In the complement of $A_n^\epsilon$ then, from equation \eqref{eq:evolutionN}, $\nod(t)$ can be bounded as follow:  

\begin{equation}\label{eq:estimateN}
	\begin{aligned} 
		\nod(t) \leq & C\int_{0}^{t}2\Big|\emq(\ped)^{\mathsf{T}} f^{(2)}\Big(\ued+\emq h\big(\euq(\ped+\tZed)\big)\Big)\Big|\\ 
		&\qquad + 2\Big|(\ped)^{\mathsf{T}}  G_\epsilon\Big(\ped+\tZed,\ued+\emq h\big(\euq(\ped+\tZed)\big),\tZed\Big)\Big| \\ 
		&\qquad + 2\Big|(\ped)^{\mathsf{T}}  R_\epsilon\Big(\euq(\ped+\tZed),\euq\ued+h\big(\euq(\ped+\tZed)\big)\Big)\Big| \\ 
		&\qquad + \frac{1}{2}\Big|\mathrm{Tr}\Big(\sigma_{hQ}^{\mathsf{T}}\sigma_{hQ}\Big)\Big| \de r +  \epsilon^\alpha.
	\end{aligned}
\end{equation}
Each term in the integral in \eqref{eq:estimateN} can be bounded as follow:

\begin{align*}
	&\begin{aligned}
		\Big|\emq(\ped(r))^{\mathsf{T}} f^{(2)}\Big(\ued(r)+\emq h\big(\euq(\ped(r)+\tZed(r))\big)\Big)\Big| \leq K_{f^{(2)}}\emq\Delta_\epsilon \|\ued(r)\|^2;
	\end{aligned}\\
	&\begin{aligned}
		\Big|(\ped(r))^{\mathsf{T}}  G_\epsilon\Big(\ped(r)+\tZed(r),  \ued(r)+\emq h\big(&\euq(\ped(r)+\tZed(r))\big),\tZed(r)\Big)\Big| \\ 
		&\leq K_G \Delta_\epsilon^2\big(\|\ped(r)\|^2+\Delta_\epsilon\|\ued(r)\|\big);
	\end{aligned}\\
	&\begin{aligned}
		\Big|(\ped(r))^{\mathsf{T}}  R_\epsilon\Big(\euq(\ped(r)+\tZed(r)),\euq\ued(r)+h\big(\euq(\ped(r)+\tZed(r))\big)\Big)\Big|\leq K_R\Delta_\epsilon^2\etqp;
	\end{aligned}\\
	&\begin{aligned}
		\Big|\mathrm{Tr}\Big(\big(\sigma_{hQ}(r)\big)^{\mathsf{T}}\sigma_{hQ}(r)\Big)\Big|\leq K_{Tr} \eum\Delta_\epsilon^2.
	\end{aligned}
\end{align*}

The constants $K_{f^{(2)}}, K_{G}$, $K_{R}$ and $K_{Tr}$ depend on the function in the subscript, but not on $\epsilon$. By the definition of $\nod$ and Theorem \ref{teo:ued} the estimate \eqref{eq:estimateN} can be rewritten

\begin{equation*}
	\begin{aligned} 
		\nod(t) \leq & K_{\max} \int_0^t \Delta_\epsilon^2 \nod(r)\de r + C_{\max}\Delta_\epsilon\Big(\emq\epsilon^\beta+\emq \epsilon^{2q}+\Delta_\epsilon^2\epsilon^q+\Delta_\epsilon \eum\Big) +  \epsilon^\alpha.
	\end{aligned}
\end{equation*}	
Using Gronwall's inequality it holds

\begin{equation*}
	\nod(t)\leq \Big[C_{\max}\Delta_\epsilon\Big(\emq\epsilon^\beta+\emq \epsilon^{2q}+\Delta_\epsilon^2\epsilon^q+\Delta_\epsilon \eum\Big) + \epsilon^\alpha \Big] e^{K_{\max}T \Delta_{\epsilon}^2}.
\end{equation*}
Observing that $e^{K_{\max}T \Delta_{\epsilon}^2}=(-\ln \epsilon)^{K_{\max}T}$, $\forall t\in[0,T\wedge\tau_{\Delta_\epsilon}^\epsilon]$, 
\begin{equation*}
	\|\ped(t) \|^2\leq \Big[C_{\max}\Delta_\epsilon\Big(\emq\epsilon^\beta+\emq \epsilon^{2q}+\Delta_\epsilon^2\epsilon^q+\Delta_\epsilon \eum\Big) + \epsilon^\alpha \Big] (-\ln \epsilon)^{K_{\max}T} \leq 2\epsilon^{2\gamma}
\end{equation*}
for a suitable choice of the parameters $\alpha$, $\beta$ and $q$. This inequality holds in the complement of $A_n^\epsilon$, which completes the proof.


\section{2-dimensional equation} \label{sec:2dimensional}
In this section the limiting behaviour of equation \eqref{eq:Z1Z2} is studied. At first its globally existence is proved. 

\begin{lemma}\label{lemma:uniquesolutionZe}
	Take $T>0$. 
	Equation \eqref{eq:Z1Z2}, with initial condition $(z_1^\epsilon,z_2^\epsilon)\in\mathbb{R}^2$, admits a unique solution in $[0,T]$. 
\end{lemma}
\begin{proof}
	The solution exists and it is unique because the coefficients are locally Lipschitz. Consider the process $K^{\epsilon}(t)=\tZpe_1(t)^2+\tZpe_2(t)^2$, with initial condition $K^{\epsilon}(0)=(z_1^\epsilon)^2+(z_2^\epsilon)^2$, then applying Ito's formula it holds 
	\begin{equation*}
		\de K^\epsilon(t) = \Big(-2K^\epsilon(t)^2+R^\epsilon\big(\euq \tZpe(t)\big)\Big)\de t + \de M^\epsilon(t),
	\end{equation*}
	where $M^\epsilon(t)$ is a martingale and 
	\begin{align*}
		R^\epsilon\big(\euq \tZpe\big) = & 2\etq\Big(\tZpe_1f^{(4)}_1\big(\euq\tZpe\big)+\tZpe_2f^{(4)}_2\big(\euq\tZpe\big)\Big)+\eum\Big(\tZpe_1 r_{q,\sigma,1}(\euq\tZpe) +\tZpe_2r_{q,\sigma,1}(\euq\tZpe)\Big)\\
		& +\sum_{ij} \Big(\sigma_{Q}'\big(\euq\tZe\big)\Big)_{ij}^2.
	\end{align*}
	Take $\Delta>0$ and define $\tau_{\Delta}^k=\{\inf t\geq 0:\, K^\epsilon(t)>\Delta \}$ then 
	\begin{align}\label{eq:ineqstoptime}
		\Expect[K^{\epsilon}(t\wedge\tau_\Delta^k)\mathbbm{1}_{\tau_\Delta^k\leq t}]\leq \Expect[K^{\epsilon}(t\wedge\tau_\Delta^k)]\leq K^{\epsilon}(0) + C_R t,
	\end{align}	
	where $C_R$ is positive constant and $\sup_{\|\tZpe\|\leq \Delta} \|R^\epsilon\big(\euq \tZpe\big) \| \leq \sup_{\euq\|\tZpe\|\leq 1} \|R^\epsilon\big(\euq \tZpe\big) \|\leq C_R$. Using the definition of the stopping time and inequality \eqref{eq:ineqstoptime} it holds 
	\begin{equation*}
		\mathsf{P}\Big(\tau_\Delta^k\leq t\Big)\leq \frac{K^\epsilon(0)+C_R t}{\Delta}\xrightarrow{\Delta\to\infty}0.
	\end{equation*}
	This shows that the solution does not explode in any finite time interval $[0,T]$.
\end{proof}

As shown in section \ref{sec:heuristic} it is convenient to pass to polar coordinates.  Take $0<\delta<N<\infty$ and  define the stopping time 
\begin{equation}\label{def:stoppingDN}
	\tilde{\tau}_{\delta, N}^{'\epsilon}:=\inf\{t>0:\ \sqrt{\tZpe_1(t)^2+ \tZpe_2(t)^2}\notin(\delta,N)\}.
\end{equation}
Consider the processes $\big(\trhoep(t),\tthetaep(t)\big)_{t\in[0,T\wedge\tilde{\tau}_{\delta, N}^{'\epsilon}]}$ defined as follow:
\begin{equation} \label{def:rhodt}
	\tZpe_1(t)=\trhoep(t)\cos\big(\tthetaep(t)\big),\quad \tZpe_2=\trhoep(t)\sin\big(\tthetaep(t)\big).
\end{equation}
Then its dynamics is described by the following SDEs

\begin{align}\label{sde:polar2_n1}
	&\begin{aligned}
		\de \trhoep = &\Big\{\frac{1}{\trhoep}\Big[-(\trhoep)^4+\frac{1}{2}\Sigma_1^2(\sin\tthetaep)^2+ \frac{1}{2}\Sigma_2^2(\cos\tthetaep)^2 +\Sigma_{12}\sin\tthetaep\cos\tthetaep\Big] \\ 
		&\quad+\euq R_{\rho,b}^\epsilon\big(\trhoep,\tthetaep\big) \Big\}\de t+\de M_\rho^\epsilon;\\
	\end{aligned}\\
	&\begin{aligned}\label{sde:polar2_n2} 
		\de \tthetaep  =& \Big\{\epsilon^{-1/2}-\frac{1}{(\trhoep)^2}\Big[\sin\tthetaep\cos\tthetaep\big(\Sigma_{1}^2-\Sigma_{2}^2\big)+\Sigma_{12}\big((\cos\tthetaep)^2-(\sin\tthetaep)^2\big)\Big] \\  
		&\quad +\euq R_{\theta,b}\big(\trhoep,\tthetaep\big)\Big\}\de t +\de M_\theta^\epsilon, 
	\end{aligned}
\end{align}
where $\Sigma_1^2$ and $\Sigma_2^2$ are as in Theorem \ref{theor:Main}, $\Sigma_{12}=\sum_{i=1}^m\bar{\sigma}_{1i}\bar{\sigma}_{2i}$ and
\begin{align*}
	M_\rho^\epsilon(t) & = \sum_{i=1}^m \int_{0}^{t}\bar{\sigma}_{1i}\cos\big(\tthetaep(r)\big)+\bar{\sigma}_{2i}\sin\big(\tthetaep(r)\big)+\euq \big(R_{	\rho,\sigma}^\epsilon\big)_i\big(\trhoep(r),\tthetaep(r)\big) \de B_i(r), \\
	M_\theta^\epsilon(t) & =\sum_{i=1}^m \int_{0}^{t} \left\{\frac{1}{\trhoep(r)}\Big[\bar{\sigma}_{2i}\cos\big(\tthetaep(r)\big)-\bar{\sigma}_{1i}\sin\big(\tthetaep(r)\big)\Big] + \euq \big(R_{	\theta,\sigma}^\epsilon\big)_i\big(\trhoep(r),\tthetaep(r)\big) \right\} \de B_i(r).
\end{align*}
The functions $R_{\rho,b}^\epsilon$ and $R_{\theta,b}^\epsilon$ include the functions $ \etq f_1^{(4)}$, $\frac{1}{2}\eum r_{q,\sigma}$ and the Ito's terms. Moreover, the functions  $R_{\rho,\sigma}^\epsilon$ and $R_{\theta,\sigma}^\epsilon$ are obtained by the Taylor expansion of $\sigma_{Q}'$ in zero and the Ito's formula.\\

The strategy of the proofs in this section follows the Strook-Varadhan approach to martingale problem, also developed in detailed in \cite{dai2019dynamics}. The first result of this section is the well-posedness of equation \eqref{eq:limit process2ndimensional}.

\begin{lemma}\label{lemma:uniquesolution} Take $T>0$. 
	Equation \eqref{eq:limit process2ndimensional} admits a unique solution in $[0,T]$, for any initial condition $\bar{\rho}_0\neq 0$.
\end{lemma}
\begin{proof} This proof is analogous to the proof of Lemma \ref{lemma:uniquesolutionZe}. Let $\big(Z_1(t),Z_2(t)\big)_{t\geq 0}$ be a solution to the following stochastic differential equation
	\begin{equation}\label{eq:Z1Z2l}
		\begin{aligned} 
			\de Z_1&=\left[-Z_1\big(Z_1^2+Z_2^2\big)\right]\de t +\sqrt{\frac{1}{2}(\Sigma_{1}^2+\Sigma_{2}^2)}\de B_1,\\ 
			\de Z_2&=\left[-Z_2\big(Z_1^2+Z_2^2\big)\right]\de t +\sqrt{\frac{1}{2}(\Sigma_{1}^2+\Sigma_{2}^2)}\de B_2,
		\end{aligned}
	\end{equation}
	where $\Sigma_1$ and $\Sigma_{2}$ are defined in Theorem \ref{theor:Main}. 
	
	With the same argument in Lemma \ref{lemma:uniquesolutionZe} one shows that equation \eqref{eq:Z1Z2l} is globally well posed.

	%

	Define
	\begin{equation*}
		\bar{\tau}_{\delta, N}=\inf\big\{t>0: \sqrt{Z_1^2+Z_2^2}\notin(\delta,N) \big\}
	\end{equation*}
	and consider the process 
	$$\bar{\rho}(t\wedge\bar{\tau}_{\delta, N})=\sqrt{Z_1(t\wedge\bar{\tau}_{\delta, N})^2+ Z_2(t\wedge\bar{\tau}_{\delta, N})^2}.$$
	Then applying Ito's formula, for any $t\in[0,T\wedge\bar{\tau}_{\delta, N}]$, it holds:
	\begin{align*}
		\de\bar{\rho}(t)=&\frac{1}{\sqrt{Z_1(t)^2+ Z_2(t)^2}}\left[-\big(Z_1(t)^2+ Z_2(t)^2\big)^2+\frac{1}{4}(\Sigma_{1}^2+\Sigma_{2}^2) \right]\de t\\ 
		&+ \sqrt{\frac{1}{2}(\Sigma_{1}^2+\Sigma_{2}^2)} \left(\frac{Z_1(t)}{\sqrt{Z_1(t)^2+ Z_2(t)^2}}\de B_1(t)+ \frac{Z_2(t)}{\sqrt{Z_1(t)^2+ Z_2(t)^2}}\de B_2(t)\right).
	\end{align*}
	We note that $\forall t\in[0,T\wedge\bar{\tau}_{\delta, N}]$, the process $\int_0^t\frac{Z_1}{\sqrt{Z_1^2+ Z_2^2}}\de B_1+\int_{0}^{t} \frac{Z_2}{\sqrt{Z_1^2+ Z_2^2}}\de B_2$ is a zero-mean martingale with quadratic variation $t$, so it is a Brownian motion. Then the process $\bar{\rho}$ is a solution to \eqref{eq:limit process2ndimensional} in the time interval $[0,T\wedge\bar{\tau}_{\delta, N}]$. Uniqueness of the solution follows because the drift and diffusion coefficients are locally Lipschitz. \\
	The proof is completed by showing that 
	\begin{equation} \label{complete}
		\lim_{\substack{\delta\to 0^+\\ N\to\infty}}\mathsf{P}(\bar{\tau}_{\delta, N}<T)=0.
	\end{equation}
	By the global existence of the process $\big(Z_1,Z_2\big)_{t\geq 0}$ then 
	$$\lim_{N\to\infty}\mathsf{P}(\sup_{0\leq t \leq T}\sqrt{Z_1(t)^2+Z_2(t)^2}>N)=0.$$
	Take a sequence $\{\delta_n\}_{n\in\mathbb{N}}$, $\delta_n\to0$, and define the event $A_n=\left\{\inf_{t\in[0,T]}\sqrt{Z_1(t)^2+Z_2(t)^2})\leq \delta_n \right\}$. To prove \eqref{complete} it is enough to show that 
	\[
	\lim_{n \rightarrow +\infty} \mathsf{P}(A_n) = 0.
	\]
	Observe that, for each $n\geq1$,  $A_{n+1}\subseteq A_{n}$, so
	\[
	\lim_{n \rightarrow +\infty} \mathsf{P}(A_n)  = \lim_{n \rightarrow +\infty} \mathsf{P}\left(\cap_{n\geq 1}A_n \right)  =  \mathsf{P}\left(\inf_{t\in[0,T]}\sqrt{Z_1(t)^2+Z_2(t)^2})= 0 \right) = 0,
	\]
	where the last identity follows for the fact that $Z_1$ and $Z_2$ form a bidimensional diffusion that is absolutely continuous w.r.t. a bidimensional Brownian motion, and the Brownian motion never visits the origin almost surely. 
	
\end{proof}
The second result states that the family of stopped processes $\trhoep(\cdot\wedge\tilde{\tau}_{\delta,N}^{'\epsilon})$ is tight. 
\begin{lemma} \label{lemma:tight}
	Define a sequence $\{\epsilon_n\}_{n\in\mathbb{N}}$, with $\epsilon_n\to 0$ as $n\to\infty$. For each $\epsilon_n$, let $\trhoepn$ be the solution to \eqref{sde:polar2_n1}, then the sequence of processes $\{\trhoepn(t\wedge\tilde{\tau}_{\delta, N}^{'\epsilon_n})_{t\in[0,T]}\}_{n\in\mathbb{N}}$ is tight. 
\end{lemma}
\begin{proof}
	The key ingredient is Aldous' criterion, see \cite{comets1988asymptotic}. The following two hypothesis need to be satisfied:
	\begin{enumerate}
		\item $\forall \gamma>0$, $\exists C>0$ such that 
		$$\sup_n\mathsf{P}\left(\sup_{t\in[0,T]}\big|\trhoepn(t\wedge\tilde{\tau}_{\delta, N}^{'\epsilon_n})\big|\geq C\right)\leq \gamma;$$
		\item $\forall \gamma>0$, and  $\forall\alpha>0$, $\exists\beta>0$ such that 
		\begin{equation}\label{eq:conditionAver}
			\sup_n\sup_{0\leq\tau_1\leq\tau_2\leq(\tau_1+\beta)\wedge T}\mathsf{P}\left( |\trhoepn(\tau_2\wedge\tilde{\tau}_{\delta, N}^{'\epsilon_n})-\trhoepn(\tau_1\wedge\tilde{\tau}_{\delta, N}^{'\epsilon_n})|\geq \alpha\right)\leq \gamma,
		\end{equation}
		where $\tau_1, \tau_2$ are stopping times adapted to the filtration generated by the process $\trhoepn$. 
	\end{enumerate}
	The first condition follows trivially because $|\trhoepn|=\trhoepn$ and $\trhoepn(t\wedge\tilde{\tau}_{\delta, N}^{'\epsilon})\leq N$.\\
	To show the second condition, we write explicitly the difference:
	\begin{align*}
		\trhoepn (\tau_2\wedge\tilde{\tau}_{\delta, N}^{'\epsilon_n}) & -\trhoepn(\tau_1\wedge\tilde{\tau}_{\delta, N}^{'\epsilon_n}) =  \int_{\tau_1\wedge\tilde{\tau}_{\delta, N}^{'\epsilon_n}}^{\tau_2\wedge\tilde{\tau}_{\delta, N}^{'\epsilon_n}}
		\Bigg\{ \frac{1}{\trhoepn(r)}\Big[-\big(\trhoepn(r)\big)^4+\frac{1}{2}\Sigma_1^2\big(\sin\tthetaepn(r)\big)^2 \\
		& + \frac{1}{2}\Sigma_2^2\big(\cos\tthetaepn(r)\big)^2 +\Sigma_{12}\sin\tthetaepn(r)\cos\tthetaepn(r)\Big] +\euq R_{\rho,b}^{\epsilon_n}\big(\trhoepn(r),\tthetaepn(r)\big) \Bigg\}\de r \\
		& + M_\rho(\tau_2\wedge\tilde{\tau}_{\delta, N}^{'\epsilon_n})-M_\rho(\tau_1\wedge\tilde{\tau}_{\delta, N}^{'\epsilon_n}).
	\end{align*}
	The term in the integral is bounded by a constant $K$ depending on $\delta$, $N$. Then
	\begin{align*}
		\big|\trhoepn (\tau_2\wedge\tilde{\tau}_{\delta, N}^{'\epsilon_n})-\trhoepn(\tau_1\wedge\tilde{\tau}_{\delta, N}^{'\epsilon_n}) \big| \leq &  K(\delta,N)\big(\tau_2\wedge\tilde{\tau}_{\delta, N}^{'\epsilon_n}-\tau_1\wedge\tilde{\tau}_{\delta, N}^{'\epsilon_n}\big) \\
		& + \big|M_\rho(\tau_2\wedge\tilde{\tau}_{\delta, N}^{'\epsilon_n})-M_\rho(\tau_1\wedge\tilde{\tau}_{\delta, N}^{'\epsilon_n})\big|.
	\end{align*}
	By the Optional Sampling Theorem:
	\begin{equation*}
		\Expect\Big[\big|M_\rho(\tau_2\wedge\tilde{\tau}_{\delta, N}^{'\epsilon_n})-M_\rho(\tau_1\wedge\tilde{\tau}_{\delta, N}^{'\epsilon_n})\big|^2 \Big]=\Expect\Big[M_\rho(\tau_2\wedge\tilde{\tau}_{\delta, N}^{'\epsilon_n})^2-M_\rho(\tau_1\wedge\tilde{\tau}_{\delta, N}^{'\epsilon_n})^2\Big].
	\end{equation*}
	Observe that
	\begin{align*}
		\Expect\Big[M_\rho(\tau_2\wedge\tilde{\tau}_{\delta, N}^{'\epsilon_n})^2\Big] & = \sum_{i=1}^m \Expect\Bigg[\int_{0}^{\tau_2\wedge\tilde{\tau}_{\delta, N}^{'\epsilon_n}}\Big\{\bar{\sigma}_{1i}^2\big(\cos\tthetaepn(r)\big)^2+\bar{\sigma}_{2i}^2\big(\sin\tthetaepn(r)\big)^2\\
		&\quad +2\bar{\sigma}_{1i}\bar{\sigma}_{2i}\cos\tthetaepn(r)\sin\tthetaepn(r) + \euq \big(\tilde{R}_{	\rho,\sigma}^\epsilon\big)_i\big(\trhoepn(r),\tthetaepn(r)\big)^2 \Big\}\de r\Bigg] \\
		&\leq (\tau_2\wedge\tilde{\tau}_{\delta, N}^{'\epsilon_n})  K_{\bar{\sigma}},
	\end{align*}
	for a  constant $K_{\bar{\sigma}}=K_{\bar{\sigma}}(\delta,N)$, where the functions $\big(\tilde{R}_{	\rho,\sigma}^\epsilon\big)_i$ contains all the terms obtained expanding the square and that are multiplied by $\euq$.  By Chebishev inequality
	\begin{equation*}
		\mathsf{P}\left(\big|M(\tau_2\wedge\tilde{\tau}_{\delta, N}^{'\epsilon_n})-M(\tau_1\wedge\tilde{\tau}_{\delta, N}^{'\epsilon_n})\big|\geq\alpha\right) \leq \frac{K_{\bar{\sigma}} \beta}{\alpha^2}.
	\end{equation*}
	In conclusion,  taking $\beta\leq\min\{\frac{\alpha^2\gamma}{K_{\bar{\sigma}}},\frac{\alpha}{K}\}$ it holds
	
	\begin{align*}
		\sup_n\sup_{0\leq\tau_1\leq\tau_2\leq(\tau_1+\beta)\wedge T}\mathsf{P}\left( |\trhoepn(\tau_2\wedge\tilde{\tau}_{\delta, N}^{'\epsilon_n})-\trhoepn(\tau_1\wedge\tilde{\tau}_{\delta, N}^{'\epsilon_n})|\geq \alpha\right)\leq \gamma.
	\end{align*}
	The second condition is proved, so the sequence of stopped processes is tight.
\end{proof}

By tightness, the sequence of processes $\{\rho^{\epsilon_n}(t\wedge\tilde{\tau}_{\delta, N}^{'\epsilon})_{t\in[0,T]}\}_{n\in\mathbb{N}}$ converges, along a subsequence, to a limit process. Call the limit process  $\rho^*_{\delta,N}$. \\  Next step consists in providing a relation between the limit by tightness and the solution to the limit equation \eqref{eq:limit process2ndimensional}.\\
Let $\bar{\mathcal{L}}$ be the generator of the limit process, solution to the equation \eqref{eq:limit process2ndimensional}, then
\begin{equation*}
	\big(\bar{\mathcal{L}}f\big)(\bar{\rho})=\bar{b}(\bar{\rho})f'(\bar{\rho})+\frac{1}{4}(\Sigma_{1}^2+\Sigma_{2}^2) f''(\bar{\rho}),
\end{equation*}
for each function $f\in\mathrm{dom}(\bar{\mathcal{L}})$ .
\begin{proposition}\label{prop:martingale}
	Let $\rho^*_{\delta,N}$ be any limit point of the sequence in Lemma \ref{lemma:tight}.
	For any $f\in C_c^\infty([\delta,N])$ the stochastic process
	\begin{equation*}
		M_{\delta,N}^{*f}(t):=f\big(\rho^*_{\delta,N}(t)\big)-f\big(\rho^*_{\delta,N}(0)\big)-\int_{0}^{t}\big(\bar{\mathcal{L}}f\big)\big(\rho^*_{\delta,N}(r)\big)\de r
	\end{equation*}
	is a martingale.
\end{proposition}

For the proof of Proposition \ref{prop:martingale} we use the following averaging principle, see \cite[Proposition~3.2]{dai2019dynamics}. 
\begin{proposition}\label{prop:tovazzi}
	Consider $\mathcal{G}: \mathbb{R} \times \mathbb{R} \rightarrow \mathbb{R}$ a locally Lipschitz continuous function, $2 \pi$-periodic in the second variable. Let $\left\{\left(\rho_n(t), \varphi_n(t)\right)_{t \in[0, T]}\right\}_{n \geq 1}$ be a family of cadlag Markov processes such that:
	\begin{enumerate}[(1)]
		\item as $n \rightarrow \infty,\left(\rho_n(t)\right)_{t \in[0, T]}$ converges, in sense of weak convergence of stochastic processes, to a process $(\rho^*(t))_{t \in[0, T]}$. Assume also that there exists a compact set $K \subset \mathbb{R}$ such that, for any $t \in[0, T]$ and $n \geq 1, \rho_n(t) \in K$ and $\rho^*(t) \in K$ and that	condition \eqref{eq:conditionAver} holds true for the sequence $\left\{\rho_n(t)\right\}_{n \geq 1}$;
		\item for any $\gamma>0$ there exist $k^{\prime}>0$ and $\bar{n} \geq 1$, such that
		$$
		\sup _{0 \leq k \leq k^{\prime}} \Expect\left[\left|\varphi_n(t+k)-\varphi_n(t)\right|\right] \leq \gamma
		$$
		for any $n \geq \bar{n}$ and $t \in[0, T]$.
	\end{enumerate}
	Then, for any $c>0$ and $\xi>0$, the following averaging principle holds:
	$$
	\int_0^T \mathcal{G}\left(\rho_n(s), c n^{\xi} s+\varphi_n(s)\right) d s \xrightarrow{\text { weakly }} \int_0^T \bar{\mathcal{G}}(\rho^*(s)) d s, \quad \text { as } n \rightarrow \infty
	$$
	where $\bar{\mathcal{G}}$ is the averaged function defined by
	$$
	\bar{\mathcal{G}}(x)=\frac{1}{2 \pi} \int_0^{2 \pi} \mathcal{G}(x, \theta) \de \theta .
	$$
\end{proposition}

The previous proposition is applied to $\rho_n(t)=\rho^{\epsilon_n}(t\wedge\tilde{\tau}_{\delta, N}^{'\epsilon})$ and to $\varphi_n(t)=\varphi^{\epsilon_n}(t\wedge\tilde{\tau}_{\delta, N}^{'\epsilon_n})$, where $\varphi^{\epsilon_n}(t):=\tthetaepn(t)-\emm_n t$. The first hypothesis is immediate to check, the second hypothesis is verified using the following proposition.

\begin{proposition}	\label{prop:average condition}
	For any $\bar{k}>0$, there exist $C>0$ and $\bar{n}\geq 1$ such that for $n \geq \bar{n}$:
	\begin{equation*}
		\sup_{0\leq k \leq\bar{k}}\Expect\Big[\big|\varphi^{\epsilon_n}\big((t+k)\wedge\tilde{\tau}_{\delta, N}^{'\epsilon_n}\big)-\varphi^{\epsilon_n}\big(t\wedge\tilde{\tau}_{\delta, N}^{'\epsilon_n}\big)\big|\Big]\leq C\sqrt{\bar{k}}.
	\end{equation*}
\end{proposition}
\begin{proof}
	By definition
	\begin{align} \nonumber
		\varphi^{\epsilon_n}\big((t+k)\wedge\tilde{\tau}_{\delta, N}^{'\epsilon_n}\big)   = & \int_{0}^{(t+k)\wedge\tilde{\tau}_{\delta, N}^{'\epsilon_n}} \Bigg\{ -\frac{1}{(\trhoepn(r))^2}\Big[\sin\tthetaepn(r) \cos\tthetaepn(r)\big(\Sigma_{1}^2-\Sigma_{2}^2\big) \\ \nonumber
		& +\Sigma_{12}\big(\cos\tthetaepn(r)\big)^2-\big(\sin\tthetaepn(r)\big)^2\big)\Big] +\euq R_{\theta,b}^{\epsilon_n}\big(\trhoepn(r),\tthetaepn(r)\big)\ \Bigg\} \de t \\ \label{eq:defphi}
		&+ M_\theta\big((t+k)\wedge\tilde{\tau}_{\delta, N}^{'\epsilon_n}\big).
	\end{align}	
	Then 
	\begin{align*}
		\varphi^{\epsilon_n}\big((t+k)\wedge\tilde{\tau}_{\delta, N}^{'\epsilon_n}\big) - \varphi^{\epsilon_n}\big(t\wedge\tilde{\tau}_{\delta, N}^{'\epsilon_n}\big)  = & \int_{t\wedge\tilde{\tau}_{\delta, N}^{'\epsilon_n}}^{(t+k)\wedge\tilde{\tau}_{\delta, N}^{'\epsilon_n}}I\big(\trhoepn(r),\tthetaepn(r)\big) \de r \\
		&+M_{\theta}\big((t+k)\wedge\tilde{\tau}_{\delta, N}^{'\epsilon_n}\big)-M_{\theta}\big(t\wedge\tilde{\tau}_{\delta, N}^{'\epsilon_n}\big),
	\end{align*}	
	where $I\big(\trhoepn(r),\tthetaepn(r)\big)$ is the function in the integral in \eqref{eq:defphi}. 
	
	By the regularity of the coefficients and the stopping time there exists a positive constant $C(\delta, N)$ such that, for each $r\in[0,T\wedge\tilde{\tau}_{\delta, N}^{'\epsilon_n}]$, the function  $I\big(\trhoepn(r),\tthetaepn(r)\big)\leq C(\delta, N)$ for every $n \geq 1$.  So
	\begin{equation} \label{ap1}
		\big|\varphi^{\epsilon_n}\big((t+k)\wedge\tilde{\tau}_{\delta, N}^{'\epsilon_n}\big)-\varphi^{\epsilon_n}\big(t\wedge\tilde{\tau}_{\delta, N}^{'\epsilon_n}\big)\big| \leq C(\delta, N)k + \left|M_{\theta}\big((t+k)\wedge\tilde{\tau}_{\delta, N}^{'\epsilon_n}\big)-M_{\theta}\big(t\wedge\tilde{\tau}_{\delta, N}^{'\epsilon_n}\big)\right|
	\end{equation}

	By the Optional Sampling Theorem it holds:
	\begin{equation*}
		\Expect\Big[\big|M_{\theta}\big((t+k)\wedge\tilde{\tau}_{\delta, N}^{'\epsilon_n}\big)-M_{\theta}\big(t\wedge\tilde{\tau}_{\delta, N}^{'\epsilon_n}\big)\big|^2 \Big] = \Expect\Big[M_{\theta}\big((t+k)\wedge\tilde{\tau}_{\delta, N}^{'\epsilon_n}\big)^2-M_\theta\big(t\wedge\tilde{\tau}_{\delta, N}^{'\epsilon_n}\big)^2\Big].
	\end{equation*}
	Observe that
	\begin{align*}
		\Expect\Big[M_\theta\big((t+k)\wedge\tilde{\tau}_{\delta, N}^{'\epsilon_n}\big)^2\Big] =  \sum_{i=1}^m\Expect\Biggl[&\int_{0}^{(t+k)\wedge\tilde{\tau}_{\delta, N}^{'\epsilon_n}}\frac{1}{\big(\trhoepn(r)\big)^2} \Big(\bar{\sigma}_{2i}\cos\tthetaepn(r)-\bar{\sigma}_{1i}\sin\tthetaepn(r) \Big)^2 \\
		& + \euq_n \big(\tilde{R}_{\theta,\sigma}^{\epsilon_n}\big)_i\big(\trhoepn(r),\tthetaepn(r)\big)\de r \Biggr],
	\end{align*}
	where the function $\big(\tilde{R}_{\theta,\sigma}^{\epsilon_n}\big)_i$ contains the terms obtained expanding the square that are multiplied by $\euq_n$. Then there exists a constant $C_{\bar{\sigma}}=C_{\bar{\sigma}}(\delta, N)$ such that 
	\begin{equation*}
		\Expect\Big[\big|M_{\theta}\big((t+k)\wedge\tilde{\tau}_{\delta, N}^{'\epsilon_n}\big)-M_{\theta}\big(t\wedge\tilde{\tau}_{\delta, N}^{'\epsilon_n}\big)\big|^2 \Big] \leq k C_{\bar{\sigma}}.
	\end{equation*}
	By Holder inequality 
	\begin{equation} \label{ap2}
		\Expect\Big[\big|M_{\theta}\big((t+k)\wedge\tilde{\tau}_{\delta, N}^{'\epsilon_n}\big)-M_{\theta}\big(t\wedge\tilde{\tau}_{\delta, N}^{'\epsilon_n}\big)\big| \Big] \leq \sqrt{C_{\bar{\sigma}}}\sqrt{k}.
	\end{equation}
	The proof is completed by \eqref{ap1}, \eqref{ap2} and the fact that 
	$\bar{k}<(\bar{k})^{\frac{1}{2}}$, for $\bar{k}<1$. 
\end{proof}

\begin{proof}[Proof of Proposition \ref{prop:martingale}]
	Take $f\in C_c^{\infty}\big([\delta,N]\big)$ and define the following generators: 
	\begin{align*}
		\big(\mathcal{L}f\big)(\rho,\theta)& =b(\rho,\theta)f'(\rho)+\frac{1}{2}\sum_{i=1}^m\big[\bar{\sigma}_{1i}\cos\theta+\bar{\sigma}_{2i}\sin\theta\big]^2	f''(\rho);\\
		\big(\mathcal{L}_{\delta,N}^{\epsilon_n} f\big)(\rho,\theta) &=\Big\{\big[b(\rho,\theta)+\euq_n R_{\rho,b}^{\epsilon_n}(\rho,\theta)\big] f'(\rho) \\
		&\qquad\qquad +\frac{1}{2} \sum_{i=1}^m \Big[\bar{\sigma}_{1i}\cos\theta +\bar{\sigma}_{2i}\sin\theta+\euq_n \big(R_{\rho,\sigma}^{\epsilon_n}\big)_i\rho\big(\rho,\theta\big) \Big]^2 f''(\rho)\Big\}\mathbbm{1}_{(\delta,N)}.
	\end{align*}
	Observe that $\big(\bar{\mathcal{L}}f\big)(\rho)=\frac{1}{2\pi}\int_0^{2\pi}\big(\mathcal{L}f\big)(\rho,\theta)\de\theta$. Define the processes
	\begin{align*}
		M_{\delta,N}^f(t)&=f\big(\trhoepn(t\wedge\tilde{\tau}_{\delta, N}^{'\epsilon_n})\big)-f\big(\trhoepn(0)\big)-\int_{0}^{t\wedge\tilde{\tau}_{\delta, N}^{'\epsilon_n}}\big(\mathcal{L}_{\delta,N}^{\epsilon_n}f\big)\big(\trhoepn(r),\tthetaepn(r)\big)\de r;\\
		N_{\delta,N}^f(t)&=f\big(\trhoepn(t\wedge\tilde{\tau}_{\delta, N}^{'\epsilon_n})\big)-f\big(\trhoepn(0)\big)-\int_{0}^{t\wedge\tilde{\tau}_{\delta, N}^{'\epsilon_n}}\big(\mathcal{L}f\big)\big(\trhoepn(r),\tthetaepn(r)\big)\de r.
	\end{align*}
	Observe that $\big(\mathcal{L}_{\delta,N}^{\epsilon_n}f\big)(\rho,\theta)=\big(\mathcal{L}f\big)(\rho,\theta)+o(\euq_n)$. By construction $M_{\delta,N}^f$ is a martingale, fix $m\geq 1$, $g_1,\ldots,g_m$ continuous bounded functions and $0\leq t_1\leq\ldots\leq t_m\leq s \leq t \leq T$, then
	\begin{equation*}
		\Expect \left[\left(M_{\delta,N}^f(t)-	M_{\delta,N}^f(s)\right)g_1\big(\trhoepn(t_1\wedge\tilde{\tau}_{\delta, N}^{'\epsilon_n})\big)\cdot\ldots\cdot g_m\big(\trhoepn(t_m\wedge\tilde{\tau}_{\delta, N}^{'\epsilon_n})\big)\right]=0.
	\end{equation*}
	Which implies
	\begin{equation}\label{eq:martingale-error}
		\Expect \left[\left(N_{\delta,N}^f(t)-	N_{\delta,N}^f(s)\right)g_1\big(\trhoepn(t_1\wedge\tilde{\tau}_{\delta, N}^{'\epsilon_n})\big)\cdot\ldots\cdot g_m\big(\trhoepn(t_m\wedge\tilde{\tau}_{\delta, N}^{'\epsilon_n})\big)\right]=o(\epsilon_n^{1/2}).
	\end{equation}
	The left hand side of equation \eqref{eq:martingale-error} can be written as 
	\begin{align}\label{eq:finitdistiP1}
		&\Expect \Big[\Big(f\big(\trhoepn(t\wedge\tilde{\tau}_{\delta, N}^{'\epsilon_n})\big)-f\big(\trhoepn(s\wedge\tilde{\tau}_{\delta, N}^{'\epsilon_n})\Big) g_1\big(\trhoepn(t_1\wedge\tilde{\tau}_{\delta, N}^{'\epsilon_n})\big)\cdot\ldots\cdot g_m\big(\trhoepn(t_m\wedge\tilde{\tau}_{\delta, N}^{'\epsilon_n})\big)\Big]\\
		&-\Expect \Big[\int_{s\wedge\tilde{\tau}_{\delta, N}^{'\epsilon_n}}^{t\wedge\tilde{\tau}_{\delta, N}^{'\epsilon_n}}\big(\mathcal{L}f\big)\big(\trhoepn(u),\tthetaepn(u)\big)\de u\ g_1\big(\trhoepn(t_1\wedge\tilde{\tau}_{\delta, N}^{'\epsilon_n})\big)\cdot\ldots\cdot g_m\big(\trhoepn(t_m\wedge\tilde{\tau}_{\delta, N}^{'\epsilon_n})\big)\Big] \label{eq:finitdistiP2}
		=o(\euq_n)
	\end{align}
	By weak convergence of $\trhoepn(t\wedge\tilde{\tau}_{\delta, N}^{'\epsilon_n})$ to $\rho^*_{\delta,N}(t)$ the term in \eqref{eq:finitdistiP1} converges to 
	$$\Expect \left[\left(f\big(\rho^*_{\delta,N}(t)\big)-f\big(\rho^*_{\delta,N}(s)\big)\right)g_1\big(\trhoepn(t_1\wedge\tilde{\tau}_{\delta, N}^{'\epsilon_n})\big)\cdot\ldots\cdot g_m\big(\trhoepn(t_m\wedge\tilde{\tau}_{\delta, N}^{'\epsilon_n})\big)\right].$$
	Since $f$ has compact support 
	\begin{equation*}
		\int_{s\wedge\tilde{\tau}_{\delta, N}^{'\epsilon_n}}^{t\wedge\tilde{\tau}_{\delta, N}^{'\epsilon_n}}\big(\mathcal{L}f\big)\big(\trhoepn(r),\tthetaepn(r)\big)\de r = \int_{s}^{t}\big(\mathcal{L}f\big)\big(\trhoepn(r\wedge\tilde{\tau}_{\delta, N}^{'\epsilon_n}),\tthetaepn(r\wedge\tilde{\tau}_{\delta, N}^{'\epsilon_n})\big)\de r.
	\end{equation*}
	Taking $\rho_n(t)=\trhoepn(t\wedge\tilde{\tau}_{\delta, N}^{'\epsilon_n})$, $\bar{\rho}(t)=\rho^*_{\delta,N}(t)$ and $\varphi_n(t)=\varphi^{\epsilon_n}(t\wedge\tilde{\tau}_{\delta, N}^{'\epsilon_n})$ proposition \ref{prop:tovazzi} can be applied to have 
	\begin{equation*}
		\int_{s}^{t}\big(\mathcal{L}f\big)\big(\trhoepn(r\wedge\tilde{\tau}_{\delta, N}^{'\epsilon_n}),\tthetaepn(r\wedge\tilde{\tau}_{\delta, N}^{'\epsilon_n})\big)\de r \to \int_{0}^{t}\big(\bar{\mathcal{L}}f\big)\big(\rho^*_{\delta,N}(r)\big)\de r.
	\end{equation*}
	By the bounded convergence theorem the proof is completed.
\end{proof}

Fix $\zeta>0$ and let $\rho^*_{\delta,N}$ be the weak limit of the sequence of stopped processes $\{\rho^{\epsilon_n}(\cdot\wedge\tilde{\tau}_{\delta, N}^{'\epsilon})\}_{n\geq 1}$. Define the stopping time \begin{equation*}
	\tau^{*}_{\delta, N,\zeta}=\inf\{t>0:\ \rho^*_{\delta,N}(t)\notin(\delta+\zeta,N-\zeta)\}.
\end{equation*}
Let $\bar{\rho}$ be the unique solution to the stochastic differential equation \eqref{eq:limit process2ndimensional} and define the stopping time 
\begin{equation*} 
	\bar{\tau}_{\delta, N,\zeta}=\inf\{t>0:\ \bar{\rho}(t)\notin(\delta+\zeta,N-\zeta)\}.
\end{equation*}

\begin{proposition}\label{prop:samelaw}
	Suppose that $\rho^*_{\delta,N}(0)=\bar{\rho}(0)$. Then, the processes $\rho^*_{\delta,N}(t\wedge\tau^{*}_{\delta, N,\zeta})$ and $\bar{\rho}(t\wedge\bar{\tau}_{\delta, N,\zeta})$ have the same distribution.
\end{proposition}
\begin{proof}
	Given $g\in C_0^\infty\big([0,\infty)\big)$ and $\zeta>0$, there exists a function $f\in C_c^\infty\big([\delta,N]\big)$ such  that $f(x)=g(x)$, $\forall x\in[\delta+\zeta,N-\zeta]$. Then
	\begin{equation*}
		g\big(\rho^*_{\delta,N}(t\wedge\tau^{*}_{\delta, N,\zeta})\big)-g\big(\rho^*_{\delta,N}(0)\big)-\int_{0}^{t\wedge\tau^{*}_{\delta, N,\zeta}}\big(\bar{\mathcal{L}}g\big)\big(\rho^*_{\delta,N}(r\wedge\tau^{*}_{\delta, N,\zeta})\big)\de r
	\end{equation*}
	is a martingale. By definition, the process $\bar{\rho}(t\wedge\bar{\tau}_{\delta, N,\zeta})$ satisfies the stopped martingale problem for $\bar{\mathcal{L}}$, in $(\delta+\zeta,N-\zeta)$. Then the stopped martingale problem is well defined and admits a unique solution. The set $C_0^\infty\big([0,\infty)\big)$ is measure-determining then  $\rho^*_{\delta,N}(t\wedge\tau^{*}_{\delta, N,\zeta})$ and $\bar{\rho}(t\wedge\bar{\tau}_{\delta, N,\zeta})$ are the solutions to the same martingale problem and by uniqueness they must have the same distribution. 
\end{proof}

The following theorem concludes this section.
\begin{theorem}\label{teo:convergencerho}
	Let $(\trhoep,\tthetaep)$ be the solution to the equations \eqref{sde:polar2_n1},\eqref{sde:polar2_n2} and initial conditions $(\rho_0^\epsilon, \theta_0^\epsilon)$, with $(\rho_0^\epsilon, \theta_0^\epsilon)\to(\rho_0, \theta_0)$, as $\epsilon\to0$. The process $\big(\trhoep(t)\big)_{t\in[0,T]}$ converges in distribution to the unique solution to the equation \eqref{eq:limit process2ndimensional} with initial condition $\rho_0$.
\end{theorem}
\begin{proof}
	Define the following sequences of positive parameters:
	\begin{itemize}
		\item $\{\delta_m\}_m$, $\{\zeta_m\}_m$ , with $r_m\to 0$ and $\zeta_m\to 0$ as $m\to\infty$;
		\item $\{N_m\}_m$, with $N_m\to\infty$ as $m\to\infty$,
	\end{itemize}
	such that $\delta_m+\zeta_m<N_m-\zeta_m$. Take the interval $I_m=(\delta_m+\zeta_m,N_m-\zeta_m)$ and, for any sequence $\{\epsilon_n\}_n$, $\epsilon_n\to0$ define the following stopping times:
	\begin{itemize}
		\item $\tau_{I_m}^{\epsilon_n}=\inf\{t>0:\ \trhoepn(t)\notin I_m\}$;
		\item $\tau^*_{I_m}=\inf\{t>0:\ \rho_{\delta_m,N_m}^*(t)\notin I_m\}$;
		\item $\bar{\tau}_{I_m}=\inf\{t>0:\ \bar{\rho}(t)\notin I_m\}$.
	\end{itemize}
	Define the event 
	$$
	A_m=\left\{x\in C\big([0,T],\mathbb{R}\big):\ x(t)\in I_m,\ \forall t\in[0,T]\right\}.
	$$
	Observe that $A_m^c$ is closed in the uniform topology of $C\big([0,T],\mathbb{R}\big)$. Then, for any $m\geq 1$ and any sequence $\epsilon_n \rightarrow 0$, by Portmanteau Theorem,
	\begin{equation*}
		\limsup_{n\to+\infty}\mathsf{P}\left(\trhoepn(\cdot\wedge\tau_{\delta_m,N_m}^{\epsilon_n})\in A_m^c\right) \leq \mathsf{P}\left(\rho_{\delta_m,N_m}^*(t)\in A_m^c\right).
	\end{equation*}
	Using proposition \ref{prop:samelaw} it holds 
	\begin{equation*}
		\mathsf{P}\big(\rho_{\delta_m,N_m}^*(\cdot)\in A_m^c\big)=\mathsf{P}\left(\tau^*_{I_m}\leq T\right)=\mathsf{P}\left( \bar{\tau}_{I_m}\leq T\right).
	\end{equation*}
	The interval $I_m$ is a subset of the interval $(\delta_m,N_m)$ then 
	\begin{equation*}
		\mathsf{P}\left(\trhoepn(\cdot\wedge\tau_{\delta_m,N_m})\in A_m^c\right)=\mathsf{P}\left(\tau_{I_m}^{\epsilon_n}\leq T\right)\leq\mathsf{P}\left( \tilde{\tau}_{\delta_m,N_m}^{'\epsilon_n}\leq T\right).
	\end{equation*}
	Then
	\begin{equation*}
		\limsup_{n\to+\infty}\mathsf{P}\left(\tilde{\tau}_{\delta_m,N_m}^{'\epsilon_n}\leq T\right) \leq \mathsf{P}\left(\tau^*_{I_m}\leq T\right)=\mathsf{P}\left( \bar{\tau}_{I_m}\leq T\right).
	\end{equation*}
	Let $f:C\big([0,T],\mathbb{R}\big)\to\mathbb{R}$ be a continuous and bounded function then
	\begin{align}
		\Bigl|\Expect\big[f(\trhoepn(\cdot))\big]-\Expect\big[f(&\bar{\rho}(\cdot))\big] \Bigr| \nonumber \\ \label{eq:i)}
		&\leq \left|\Expect\big[f(\trhoepn(\cdot))\big]-\Expect\big[f(\trhoepn(\cdot\wedge\tilde{\tau}_{\delta_m,N_m}^{'\epsilon_n}))\big]\right| \\\label{eq:ii)}
		&\ + \left|\Expect\big[f(\trhoepn(\cdot\wedge\tilde{\tau}_{\delta_m,N_m}^{'\epsilon_n}))\big]-\Expect\big[f(\rho_{\delta_m,N_m}^*(\cdot))\big]\right| \\\label{eq:iii)}
		&\ + \Bigl|\Expect\big[f(\rho_{\delta_m,N_m}^*(\cdot))\big]-\Expect\big[f(\rho_{\delta_m,N_m}^*(\cdot\wedge\tau^*_{I_m}))\big]\Bigr| \\	\label{eq:iv)}
		&\ + \Bigl|\Expect\big[f(\rho_{\delta_m,N_m}^*(\cdot\wedge\tau^*_{I_m}))\big]- \Expect\big[f(\bar{\rho}(\cdot\wedge\bar{\tau}_{I_m}))\big]\Bigr| \\	\label{eq:v)}
		&\ + \Bigl|\Expect\big[f(\bar{\rho}(\cdot\wedge\bar{\tau}_{I_m}))\big] - \Expect\big[f(\bar{\rho}(\cdot))\big]\Bigr|.
	\end{align}
	Let $\|f\|_\infty$ be the sup norm of the function $f$, then
	\begin{itemize}
		\item \eqref{eq:i)}$\leq \|f\|_\infty \mathsf{P}\left( \tilde{\tau}_{\delta_m,N_m}^{'\epsilon_n}\leq T\right)$;
		\item \eqref{eq:iii)}$\leq \|f\|_\infty \mathsf{P}\left( \tau^*_{I_m}\leq T\right)$;
		\item \eqref{eq:iv)}$=0$ by Proposition \ref{prop:samelaw};
		\item \eqref{eq:v)}$\leq \|f\|_\infty \mathsf{P}\left( \bar{\tau}_{I_m}\leq T\right)$.
	\end{itemize}
	By Lemma \ref{lemma:uniquesolution}, for any $\gamma>0$, there is $m$ large enough such that $\mathsf{P}\left( \bar{\tau}_{I_m}\leq T\right)\leq \gamma$. Moreover, by convergence in distribution, there is $\epsilon_n$ small enough such that \eqref{eq:ii)}$\leq\gamma$ and $\mathsf{P}\left( \tilde{\tau}_{\delta_m,N_m}^{'\epsilon_n}\leq T\right)\leq 2 \gamma$. The proof is complete.
\end{proof}

\section{Proof of Theorem \ref{theor:Main} }\label{sec:proofMain}

In this section the results showed in the previous sections are used to prove Theorem \ref{theor:Main}.\\

\begin{proof}
	Define the following stopping times
	\begin{align*}
		&\tilde{\tau}^{z'}_{\Delta_\epsilon}=\inf\{t\in[0,T]:\, \|\tZpe(t) \|>\Delta_\epsilon\};\\
		&\tilde{\tau}^{z}_{\Delta_\epsilon}=\inf\{t\in[0,T]:\, \|\tZe(t) \|>\Delta_\epsilon\};\\
		&\tau^{z'}_{\Delta_\epsilon}=\inf\{t\in[0,T]:\, \|\Zpe(t) \|>\Delta_\epsilon\};\\
		&\tau^{y'}_{\Delta_\epsilon}=\inf\{t\in[0,T]:\, \|\Ype(t) \|>\Delta_\epsilon\};\\
		&\tau^x_{\Delta_\epsilon}=\inf\{t\in[0,T]:\, \|\big(\Ze(t),\Ye(t)\big)^{\mathsf{T}} \|>\Delta_\epsilon\}.
	\end{align*}
	From Lemma \ref{lemma:uniquesolutionZe} it follows that $\lim_{\epsilon\to 0}\mathsf{P}\big(\tilde{\tau}^{z'}_{\Delta_\epsilon}\leq T\big)=0$. By definition of the transformation $\mathfrak{q}$, see \eqref{diff:q}, and
	$\tilde{\tau}^{z}_{\Delta_\epsilon}=\inf\{t\in[0,T]:\, \|\tZpe(t) + \eum q\big(\tZpe(t)\big)\|>\Delta_\epsilon\}$ then, if $\epsilon$ is small enough,
	\begin{equation*}
		\mathsf{P}\big(\tilde{\tau}^{z}_{\Delta_\epsilon}\leq T\big)\leq \mathsf{P}\big(\tilde{\tau}'_{\Delta_\epsilon/2}\leq T\big)\xrightarrow{\epsilon\to 0} 0.
	\end{equation*}
	Moreover $\lim_{\epsilon\to 0}\mathsf{P}\big(\tau^{z'}_{\Delta_\epsilon}\leq T\big)=0$. Indeed
	\begin{align*}
		\mathsf{P}\big(\sup_{t\in[0,\tau^{z'}_{\Delta_\epsilon}\wedge T]}\|\Zpe(t)\|>\Delta_\epsilon\big) \leq &  \mathsf{P}\big(\sup_{t\in[0,\tau^{z'}_{\Delta_\epsilon}\wedge T]}\|\Zpe(t)-\tZe(t)\|>\frac{\Delta_\epsilon}{2}\big) \\
		&+ \mathsf{P}\big(\sup_{t\in[0, T]}\|\tZe(t)\|>\frac{\Delta_\epsilon}{2}\big)\xrightarrow{\epsilon\to 0} 0,
	\end{align*}
	by Theorem \ref{teo:ped}. Using an analogous argument and Theorem \ref{teo:ued}, it holds $\lim_{\epsilon\to 0}\mathsf{P}\big(\tau^{y'}_{\Delta_\epsilon}\leq T\big)=0$ and then $\lim_{\epsilon\to 0}\mathsf{P}\big(\tau_{\Delta_\epsilon}^\epsilon\leq T\big)=0$, where $\tau_{\Delta_\epsilon}^\epsilon$ is defined in \eqref{def:stoppingZYZ}.  By definition of the transformation $\mathfrak{p}$, see \eqref{diff:p}, if $\epsilon$ is small enough,
	\begin{equation}\label{eq:tauXestimate}
		\begin{aligned}
			\mathsf{P}\big(\tau^x_{\Delta_\epsilon}\leq T\Big) = \mathsf{P}\Big(\inf\{t\in[0,T]:&\, \|\begin{pmatrix*}\Zpe(t) + \euq p\big(\Zpe(t)\big) \\ \Ype(t) \end{pmatrix*}\|>\Delta_\epsilon\}\leq T\Big) \\
			& \leq  \mathsf{P}\big(\tau^{z'}_{\Delta_\epsilon/2}\leq T\big) + \mathsf{P}\big(\tau^{y'}_{\Delta_\epsilon/2}\leq T\big).
		\end{aligned}
	\end{equation}
	Define the processes 
	\begin{equation*}
		\rhoe(t)= \sqrt{\Ze_1(t)^2+\Ze_2(t)^2},\qquad  \rhoep(t)= \sqrt{\Zpe_1(t)^2+\Zpe_2(t)^2},\qquad \trhoe(t) = \sqrt{\tZe_1(t)^2+\tZe_2(t)^2}.
	\end{equation*}
	
	Let $f:C\big([0,T],\mathbb{R}\big)\to\mathbb{R}$ be a continuous and bounded function then
	\begin{equation}\label{eq:finalineq}
		\begin{aligned}
			\Bigl|\Expect\big[f(\rhoe(\cdot))\big]-\Expect\big[f(\bar{\rho}(\cdot))\big] \Bigr| \leq & \left|\Expect\big[f(\rhoe(\cdot))\big]-\Expect\big[f(\rhoe(\cdot\wedge\tau^x_{\Delta_\epsilon}))\big]\right| \\
			&\ + \left|\Expect\big[f(\rhoe(\cdot\wedge\tau^x_{\Delta_\epsilon}))\big]-\Expect\big[f(\rhoep(\cdot\wedge\tau_{\Delta_\epsilon}^\epsilon))\big]\right| \\
			&\ + \Bigl|\Expect\big[f(\rhoep(\cdot\wedge\tau_{\Delta_\epsilon}^\epsilon))\big]-\Expect\big[f(\trhoe(\cdot\wedge\tau_{\Delta_\epsilon}^\epsilon)\big]\Bigr| \\	
			&\ + \Bigl|\Expect\big[f(\trhoe(\cdot\wedge\tau_{\Delta_\epsilon}^\epsilon))\big]- \Expect\big[f(\trhoep(\cdot\wedge\tilde{\tau}^{z'}_{\Delta_\epsilon}))\big]\Bigr| \\		
			&\ + \Bigl|\Expect\big[f(\trhoep(\cdot\wedge\tilde{\tau}^{z'}_{\Delta_\epsilon})\big] - \Expect\big[f(\trhoep(\cdot)\big]\Bigr|\\
			&\ + \Bigl|\Expect\big[f(\trhoep(\cdot)\big] - \Expect\big[f(\bar{\rho}(\cdot))\big]\Bigr|\xrightarrow{\epsilon\to 0} 0.
		\end{aligned}
	\end{equation}
	The convergence follows from Lemma \ref{lemma:uniquesolution}, $\lim_{\epsilon}\sup_{t\in[0,T]}\big|\rhoe(t\wedge\tau^x_{\Delta_\epsilon}))-\rhoep(t\wedge\tau_{\Delta_\epsilon}^\epsilon)\big|=0$, Theorem \ref{teo:ped}, $\lim_{\epsilon}\sup_{t\in[0,T]}\big|\trhoe(t\wedge\tau_{\Delta_\epsilon}^\epsilon))-\trhoep(t\wedge\tilde{\tau}^{z'}_{\Delta_\epsilon})\big|=0$, inequality \eqref{eq:tauXestimate} and Theorem \ref{teo:convergencerho}. 
\end{proof}

\section{Appendix}\label{appendix}

\subsection{Normal forms} \label{section:Normalform}
The function $f$ in equation \eqref{eq:Carr1} is at least quadratic in $z$ and $y$ close to the origin. In this section is showed that there is a smooth change of coordinates such that quadratic terms cancel out. Suppose that matrix $Q$ has the following expression:
\begin{equation*}
	Q=\begin{pmatrix*}0 & -\lambda_0 \\ \lambda_0 & 0\end{pmatrix*}
\end{equation*}
with $\lambda_0>0$. According to the theory of normal form, see \cite{wiggins1996introduction}, the first step consists in passing to complex coordinate using the linear transformation $\mathfrak{l}:\mathbb{R}^2\times\mathbb{R}^{n-2}\to\mathbb{C}^2\times\mathbb{R}^{n-2}$, $\mathfrak{l}(z,y)=(w,\bar{w},y)$, where
\begin{equation}\label{def:changecomplex}
	\begin{pmatrix}
		w \\
		\bar{w}
	\end{pmatrix} 
	= \begin{pmatrix}
		1 & i \\ 1 & -i 
	\end{pmatrix} \begin{pmatrix}
		z_1 \\
		z_2
	\end{pmatrix} 
\end{equation}
and $\bar{w}\in\mathbb{C}$ is the complex conjugate of $w\in\mathbb{C}$. Equation \eqref{eq:Carr1} can be written in the new variables $(w, \bar{w}, y)= \mathfrak{l}\Big(z, y\Big)$:
\begin{align*}
	\de  w=& \left[ \lambda_0i w+\Big(f_1\big( w, \bar{w}, y\big)+if_2\big( w,\bar{w},y\big)\Big) \right]\de t, \\
	\de  \bar{w}=& \left[-\lambda_0i \bar{w}+\Big(f_1\big( w, \bar{w}, y\big)-if_2\big( w,\bar{w}, y\big)\Big) \right]\de t,
\end{align*}
where $(f_1,f_2)^\mathsf{T}=f$ and with an abuse of notation $f(w,\bar{w},y)=f(z,y)$.\\
Define the transformation $\mathfrak{R}:\mathbb{C}^2\times\mathbb{R}^{n-2}\to\mathbb{C}^2\times\mathbb{R}^{n-2}$
\begin{equation*}
	\mathfrak{R}(w',\bar{w}',y) = 
	\begin{pmatrix*}[c]
		w' +\mathfrak{r}(w',\bar{w}',y) \\
		\bar{w}' +\overline{\mathfrak{r}}(w',\bar{w}',y) \\
		y
	\end{pmatrix*}=	\begin{pmatrix*}[c]
		w \\
		\bar{w}\\
		y
	\end{pmatrix*}
\end{equation*}
where the function $\mathfrak{r}(w,\bar{w},y)=\beta_1w^2+\beta_2\bar{w}^2+\beta_{12}w\bar{w}+\sum_{i=1}^{n-2}\alpha_{1i}wy_i+\alpha_{2i}\bar{w}y_i$, for $\beta_{j},\beta_{12},\alpha_{ji}$ real parameters, and $\overline{\mathfrak{r}}$ is the complex conjugate of $\mathfrak{r}$. Let $\mathrm{Id}$ be the identity matrix, then

\begin{equation*}
	\mathrm{D}\mathfrak{R}(w',\bar{w}',y)=\begin{pmatrix*}1+\partial_{w}\mathfrak{r} & \partial_{\bar{w}}\mathfrak{r} & \nabla_{y}\mathfrak{r} \\
		\partial_{w}\mathfrak{r} & 1+\partial_{\bar{w}}\mathfrak{r} & \nabla_{y}\mathfrak{r} \\
		0 & 0 & \mathrm{Id} \end{pmatrix*} = \mathrm{Id} + A,
\end{equation*}
where
\begin{equation*}
	A=	\begin{pmatrix*} \partial_{w}\mathfrak{r} & \partial_{\bar{w}}\mathfrak{r} & \nabla_{y}\mathfrak{r} \\
		\partial_{w}\mathfrak{r} & \partial_{\bar{w}}\mathfrak{r} & \nabla_{y}\mathfrak{r} \\
		0 & 0 & 0\end{pmatrix*}.
\end{equation*}
If $(w',\bar{w}',y)$ is close to zero, define $\mathfrak{R}^{-1}$ as the inverse function of $\mathfrak{R}$, then

\begin{equation*}
	\de w'=D\mathfrak{R}^{-1}(w,\bar{w},y)\begin{pmatrix*}\de w\\ \de\bar{w}\\\de y\end{pmatrix*}.
\end{equation*}
Moreover, 
\begin{equation*}
	D\mathfrak{R}^{-1}(w,\bar{w},y)=\big(D\mathfrak{R}(w',\bar{w}',y)\big)^{-1}=\mathrm{Id} - A + \frac{1}{2}A^2 + R_A,
\end{equation*}
where $R_A$ is the rest of the Taylor expansion in zero. Call
\begin{equation*}
	F_{\pm}(w,\bar{w},y)= f_1\big( w, \bar{w}, y\big)\pm if_2\big( w,\bar{w},y\big)
\end{equation*}
and
\begin{align*}
	\de w' =& \Big[ 
	\big(\lambda_0i\w+F_{+}(w,\bar{w},y)\big)\big(1-\partial_{w}\mathfrak{r}+\frac{1}{2}(\partial_{w}\mathfrak{r})^2+\frac{1}{2}\partial_{w}\mathfrak{r}\partial_{\bar{w}}\mathfrak{r}\big)\\
	&\quad+\big(-\lambda_0i\bar{w}+F_{-}(w,\bar{w},y)\big)\big(-\partial_{\bar{w}}\mathfrak{r} + \frac{1}{2}(\partial_{\bar{w}})^2+\frac{1}{2}\partial_{w}\mathfrak{r}\partial_{\bar{w}}\mathfrak{r}\big)\\ 
	&\quad+ \big(Py+g(w,\bar{w},y)\big)\big(-\nabla_{y}\mathfrak{r} + \nabla_{y}\mathfrak{r}\partial_{w}\mathfrak{r}+\nabla_{y}\mathfrak{r}\partial_{\bar{w}}\mathfrak{r}\big) \\
	&\quad+r_A(w',\bar{w}',y)	\Big] \de t,
\end{align*}
where $\lim_{\|(w,\bar{w},y)\|\to 0}\frac{|r_A(w,\bar{w},y)|}{\|(w,\bar{w},y)\|^3}=0$.  Define the function  $ F_{+}'(w',\bar{w}',y)=F_{+}\circ \mathfrak{R}(w',\bar{w}',y)$ and write
\begin{equation*}
	F_{+}'(w',\bar{w}',y) = F_y^{(2)}(y) +F_+^{(2)}(w',\bar{w}',y) +F_+^R(w',\bar{w}',y),
\end{equation*}
where $F_{y}^{(2)}$ is quadratic in the variable $y$,  and the functions $F_+^{(2)}$, $F_+^R$ contain the other quadratic terms and  the higher order terms, respectively. Then 
\begin{align}
	\de w' =& \Big[ \lambda_0i\w' + F_y^{(2)} +\lambda_0 i \mathfrak{r} - \lambda_0 i \partial_{w}\mathfrak{r} w' + \lambda_0 i \partial_{\bar{w}}\mathfrak{r} \bar{w}' - \nabla_{y}\mathfrak{r}Py +F_+^{(2)}(w',\bar{w}',y)\nonumber \\ \label{eq:wprime}
	& + F^{(3)}(w',\bar{w}',y)+r_A'(w',\bar{w}',y)	\Big] \de t,
\end{align}
where the functions $F^{(3)}$ and $r_A'$ contain the third order terms and higher order terms.

\begin{proposition}\label{prop:solutionp}
	Take $\lambda_0\in\mathbb{R}$, $P\in\mathbb{R}^{n-2\times n-2}$, and $F_{+}^{(2)}:\mathbb{C}^2\times\mathbb{R}^{n-2}\to\mathbb{C}$ as in equation \eqref{eq:wprime}. Then there exists a function $\mathfrak{r}:\mathbb{C}^2\times\mathbb{R}^{n-2}\to\mathbb{C}$, where   $\mathfrak{r}(w,\bar{w},y)=\beta_1w^2+\beta_2\bar{w}^2+\beta_{12}w\bar{w}+\sum_{i=1}^{n-2}\alpha_{1i}wy_i+\alpha_{2i}\bar{w}y_i$, for $\beta_{j},\beta_{12},\alpha_{ji}$ real parameters such that
	\begin{equation}\label{eq:eta}
		\lambda_0i\mathfrak{r} - \lambda_0i\partial_{w}\mathfrak{r}w +\lambda_0i\partial_{\bar{w}}\mathfrak{r}\bar{w}-\nabla_{y}\mathfrak{r}P\y  =- F_+^{(2)}.
	\end{equation} 
\end{proposition}
\begin{proof}
	Take a polynomial $\mathfrak{r}$ and consider the map $L:\mathfrak{r}\xrightarrow{L} L\mathfrak{r}$ where 
	\begin{equation*}
		L\mathfrak{r}=\lambda_0i\mathfrak{r} - \lambda_0i\partial_{w}\mathfrak{r}w +\lambda_0i\partial_{\bar{w}}\mathfrak{r}\bar{w}-\nabla_{y}\mathfrak{r}P\y.
	\end{equation*} 
	The map $L$ is linear in the space of homogeneous polynomial in the variables $w,\bar{w},(y_j)_{j=1}^{n-2}$ of degree two into itself. Denote $H_2=\textrm{span}\big\{w^2,\bar{w}^2,w\bar{w}, (wy_j)_{j=1}^{n-2},$ $(\bar{w}y_j)_{j=1}^{n-2}\big\}$ and compute the action of the map $L$ on the canonical basis:
	\begin{align*}
		&L(w^2)=-\lambda_0iw^2,\\
		&L(\bar{w}^2)=3\lambda_0i\bar{w}^2,\\
		&L(w\bar{w})=\lambda_0iw\bar{w},\\
		&L(wy_j)=-w\sum_{l=1}^{n-2}P_{jl}y_l,\\
		&L(\bar{w}y_j)=\bar{w}\big(2\lambda_0iy_j-\sum_{l=1}^{n-2}P_{jl}y_l\big).
	\end{align*}
	The matrix associated to the linear application in the canonical basis 
	$$\big\{w^2,\bar{w}^2, w\bar{w}, (wy_j)_{j=1}^{n-2}, (\bar{w}y_j)_{j=1}^{n-2}\big\}$$
	is
	\begin{equation*}
		\small
		L=\begin{pmatrix}
			-\lambda_0i & 0          & 0          & 0                & \ldots & 0         & 0        & \ldots & 0 \\
			0           & 3\lambda_0i & 0          & 0                 & \ldots & 0         & 0        & \ldots & 0 \\
			0           & 0          & \lambda_0i & 0                & \ldots & 0         & 0        & \ldots & 0 \\
			0           & 0          & 0          & -P_{1,1}     & \ldots & -P_{1,n-2} & 0        & \ldots & 0 \\
			0           & 0          & 0          & -P_{2,1}     & \ldots & -P_{2,n-2} & 0        & \ldots & 0 \\
			\ldots      & \ldots     & \ldots     & \ldots      & \ldots & \ldots    & \ldots   & \ldots & \ldots  \\	
			0           & 0          & 0          & -P_{n-2,1}  & \ldots & -P_{n-2,n-2} & 0        & \ldots & 0 \\	
			0           & 0          & 0          & 0                  & \ldots & 0      & 2\lambda_0i-P_{1,1}      & \ldots & -P_{1,n-2} \\	
			\ldots      & \ldots     & \ldots     & \ldots      & \ldots & \ldots    & \ldots   & \ldots & \ldots  \\	
			0           & 0          & 0          & 0                   & \ldots & 0      & -P_{n-2,1}      & \ldots & 2\lambda_0i-P_{n-2,n-2} 
		\end{pmatrix}.
	\end{equation*}
	Let $D=\left(\begin{smallmatrix*}-\lambda_0i & 0 & 0\\ 0 & 3\lambda_0i & 0\\0& 0 & \lambda_0i \end{smallmatrix*}\right)$ then the matrix $L$ can be written as: 
	\begin{equation*}
		L=\begin{pmatrix}
			D & 0  & 0 \\
			0 & -P & 0 \\
			0 & 0 &  2\lambda_0i \Id-P
		\end{pmatrix}.
	\end{equation*}
	Observe that $\det A =\det D \det (-P) \det (2\lambda_0iId-P)\neq 0$. Indeed suppose that $\det (-2\lambda_0i\Id-P)=0$ then $0$ is an  eigenvalue of the matrix $2\lambda_0i\Id-P$ and there is a vector $v\in\textrm{span}\{\bar{w}y_l,\ j=1,\ldots,n-2\}$ such that $\big(2\lambda_0i\Id-P\big)v=0$. This implies that  $2\lambda_0i$ is eigenvalue of $P$, but $P$ admits only eigenvalues with strictly negative real part. Matrix $L$ is invertible then there is a unique solution $\mathfrak{r}$ to the equation $L\mathfrak{r}=- F_+^{(2)}$ and \eqref{eq:eta} holds.
\end{proof}
Choose $\mathfrak{r}$ as in Proposition \ref{prop:solutionp} then equation \eqref{eq:wprime} becomes
\begin{align*}
	\de w' =& \Big[ -\lambda_0i\w' + F_y^{(2)} + F^{(3)}(w',\bar{w}',y)+r_A'(w',\bar{w}',y)	\Big] \de t.
\end{align*}
Applying the inverse function $\mathfrak{l}^{-1}$, the equation \eqref{eq:Carr1} becomes
\begin{align}
	\de z = & \Big[ Qz+f^{(3)}(z,y)+f_y^{(2)}(y) + r_f(z,y)  \Big]\de t
\end{align}
where $f^{(3)}$ is the function containing all the cubic terms, $f_y^{(2)}$ is quadratic in $y$ and the higher order terms are in $r_f$.
\paragraph*{Acknowledgment:} The authors are very grateful to professor Lorenzo Bertini for introducing them to the problem and for the inspiring discussions.


\end{document}